\documentclass[11pt]{article}

\usepackage[utf8]{inputenc}
\IfFileExists{lmodern.sty}{
  \usepackage[T1]{fontenc}
  \usepackage{lmodern}
}{
  \usepackage[OT1]{fontenc}
}
\usepackage{microtype}
\usepackage[a4paper,margin=28mm]{geometry}
\usepackage{mathtools,amssymb,amsthm}
\usepackage{aliascnt}
\usepackage{booktabs,array}
\usepackage{enumitem}
\usepackage{float}
\usepackage{xurl}
\usepackage[hidelinks]{hyperref}
\usepackage[nameinlink,noabbrev]{cleveref}

\numberwithin{equation}{section}

\newtheorem{theorem}{Theorem}[section]
\newaliascnt{proposition}{theorem}
\newtheorem{proposition}[proposition]{Proposition}
\aliascntresetthe{proposition}
\newaliascnt{lemma}{theorem}
\newtheorem{lemma}[lemma]{Lemma}
\aliascntresetthe{lemma}
\newaliascnt{corollary}{theorem}
\newtheorem{corollary}[corollary]{Corollary}
\aliascntresetthe{corollary}
\theoremstyle{definition}
\newaliascnt{definition}{theorem}

\aliascntresetthe{definition}
\theoremstyle{remark}
\newaliascnt{remark}{theorem}

\aliascntresetthe{remark}

\crefname{theorem}{theorem}{theorems}
\crefname{proposition}{proposition}{propositions}
\crefname{lemma}{lemma}{lemmas}
\crefname{corollary}{corollary}{corollaries}
\crefname{definition}{definition}{definitions}
\crefname{remark}{remark}{remarks}

\newcommand{\F}{\mathcal F}
\newcommand{\A}{\mathcal A}
\newcommand{\B}{\mathcal B}
\newcommand{\C}{\mathcal C}
\newcommand{\D}{\mathcal D}
\newcommand{\E}{\mathcal E}
\newcommand{\Hh}{\mathcal H}
\newcommand{\K}{\mathcal K}
\newcommand{\X}{\mathcal X}
\newcommand{\CH}{\operatorname{CH}}
\newcommand{\Lk}{\operatorname{Lk}}
\newcommand{\mc}{\operatorname{mc}}
\newcommand{\mcz}{\operatorname{mc}_0}
\newcommand{\supp}{\operatorname{supp}}
\newcommand{\Z}{\mathbb Z}
\newcommand{\dd}{\mathbin{\dot\cup}}

\title{Sunflower-Free Uniform Families: Recursive Constructions and Explicit Bounds}
\author{%
\small Edward Axante\thanks{Faculty of Mathematics and Computer Science,
 University of Bucharest, Romania.}, Cristian Budala\footnotemark[1],
 David Chitic\footnotemark[1], Bogdan Dumitru\footnotemark[1]\hspace{0.65em}and
 Mihai Nacu\footnotemark[1]%
}
\date{}

\begin{document}

\maketitle

\begin{abstract}
Let $f(w,k)$ be the maximum size of a $w$-uniform family containing no
sunflower with $k$ petals. We introduce a recursive construction for
sunflower-free families and use it to obtain a general lower bound on the
exponential growth rate of $f(w,k)$. We also prove a general upper bound for
$3$-uniform families with at least four petals. Our results give
$39\le f(3,4)\le49$, $f(3,5)\le146$,
$153\le f(3,6)\le255$, $259\le f(3,7)\le474$, and
$54\le f(4,3)\le83$. In addition, we prove that the maximum size of an
intersecting $4$-uniform family containing no sunflower with three petals is $27$.
The upper bounds $49$ and $83$ are computer-assisted. The finite lower
bounds come from explicit constructions.
\end{abstract}

\section{Introduction}
\label{sec:introduction}

A sunflower with $k$ petals, or $k$-sunflower, is a collection of $k$ distinct sets whose
pairwise intersections are equal. Their common intersection is called the
core and may be empty. Let $f(w,k)$ denote the maximum size of a family of
distinct $w$-element sets containing no $k$-sunflower. The
sunflower conjecture of Erd\H{o}s and Rado asks whether, for every fixed $k$,
there is a constant $C_k$ such that $f(w,k)\le C_k^w$ for every $w$
\cite{ErdosRado1960}. Alweiss, Lovett, Wu, and Zhang substantially improved
the classical upper bound \cite{AlweissLovettWuZhang2021}, and Bell,
Chueluecha, and Warnke subsequently proved
$f(w,k)\le (Ck\log w)^w$ for an absolute constant $C$ and all $w,k\ge2$
\cite{BellChueluechaWarnke2021}.

Our lower-bound arguments keep track not only of the absence of sunflowers but
also of the maximum number of pairwise disjoint members. Controlling this
matching number rules out sunflowers with empty core. Within this framework,
we construct larger sunflower-free families from smaller ones by attaching
subsets of an auxiliary set. The resulting recurrence leads to an explicit
general lower bound on the exponential growth rate of $f(w,k)$. An
alternative construction based on a bounded-degree auxiliary graph is
included in the appendix.

Abbott and Hanson introduced the study of sunflower-free families with a
bounded matching number and showed how families on disjoint ground sets can be
combined to obtain lower bounds for $f(w,k)$ \cite{AbbottHanson1974}. For
three petals, our main recurrence follows from a multiplication theorem of
Abbott and Exoo \cite{AbbottExoo1992}. We extend this recurrence to arbitrary
numbers of petals and all admissible matching-number bounds. To our knowledge,
this recurrence, the alternative recurrence in the appendix, and the resulting
exponential lower bound are new. Abbott and Hanson raised the problem of
proving recurrences for $\psi$ and reported no success
\cite[p.~11]{AbbottHanson1974}. The closest prior forms are an additive
recurrence of Abbott and Exoo for three petals
\cite[Theorem~3]{AbbottExoo1992} and an edge-attachment product of Frankl and
Wang \cite[Proposition~2.1]{FranklWang2025}. For $k-1\ge3w(m-1)$, Frankl
and Wang determine $\psi(w,k,m)$ exactly \cite[Theorem~1.11]{FranklWang2025};
the bounds proved here concern fixed $k$ and growing $w$, outside that
regime.

For upper bounds on families of triples, we begin with a largest collection of
pairwise disjoint triples and classify all other members by how they intersect
its union. This reduces the problem to degree and matching bounds for ordinary
graphs, to which we apply the theorem of Chv\'atal and Hanson
\cite{ChvatalHanson1976}. The resulting argument gives a general upper bound
for $f(3,k)$ when $k\ge4$. Separate finite analyses sharpen the result for
$f(3,4)$ and give the upper bound for $f(4,3)$.

\Cref{tab:intro-bounds} compares the numerical bounds obtained in this paper
with previously published bounds. The two published endpoints shown for
$f(3,7)$ follow from Proposition~2.1 and Theorem~1.14, respectively, of Frankl
and Wang \cite{FranklWang2025}.

\begin{table}[H]
\centering
\caption{Published comparisons and bounds proved in this paper.}
\label{tab:intro-bounds}
\begin{tabular}{@{}c@{\quad}r@{\;}c@{\;}l@{\qquad}r@{\;}c@{\;}l@{}}
\toprule
Quantity & \multicolumn{3}{c}{Published comparison}
         & \multicolumn{3}{c}{Bound proved here}\\
\midrule
$f(3,4)$ & $38$ \cite{AbbottExoo1992} $\le f(3,4)$
         & $\le$ & $69$ \cite{AbbottHansonSauer1972}
         & $39\le f(3,4)$ & $\le$ & $49$\\
$f(3,5)$ & $f(3,5)$ & $\le$ & $180$ \cite{AbbottHansonSauer1972}
         & $f(3,5)$ & $\le$ & $146$\\
$f(3,6)$ & $146$ \cite{AbbottExoo1992} $\le f(3,6)$
         & $\le$ & $305$ \cite{AbbottHansonSauer1972}
         & $153\le f(3,6)$ & $\le$ & $255$\\
$f(3,7)$ & $252\le f(3,7)$ & $\le$ & $498$ \cite{FranklWang2025}
         & $259\le f(3,7)$ & $\le$ & $474$\\
$f(4,3)$ & $54$ \cite{AbbottHanson1974} $\le f(4,3)$
         & $\le$ & $142$ \cite{AbbottHansonSauer1972}
         & $54\le f(4,3)$ & $\le$ & $83$\\
\bottomrule
\end{tabular}
\end{table}

The upper bounds $f(3,4)\le49$ and $f(4,3)\le83$ are computer-assisted and
follow from exhaustive finite enumerations. The lower bounds
$f(3,4)\ge39$ and $f(3,7)\ge259$ are established by explicit finite families.
We also determine exactly the maximum size of an intersecting $4$-uniform
family containing no $3$-sunflower: it is $27$. The main
combinatorial arguments of
\cref{sec:preliminaries,sec:constructions,sec:triple-upper}, the appendix
capacity recurrence, and every explicit finite lower-bound witness have
additionally been formalised and machine-checked in the Lean~4 proof assistant
(\cref{sec:lean}).\footnote{Code, data, and the Lean development are available at
\url{https://github.com/bogdan27182/sunflower-paper}.}

\section{Notation and preliminary facts}
\label{sec:preliminaries}

This section fixes the notation and records the basic facts used in
the paper. We introduce links and the matching-constrained extremal quantity,
determine the cases of uniformity one and two, and recall the product
inequality used to define the exponential growth rate.

\subsection{Set families}

All set families in this paper are finite and have distinct members. All
graphs are finite and simple. For a finite set $X$, write
\[
 \binom{X}{w}=\{A\subseteq X:|A|=w\}.
\]
A family $\F$ is $w$-uniform if $\F\subseteq\binom{X}{w}$ for some finite
set $X$. Its support is
\[
 \supp(\F)=\bigcup_{F\in\F}F.
\]
If distinct sets $A_1,\ldots,A_k$ satisfy
\[
 A_i\cap A_j=C\qquad(1\le i<j\le k),
\]
then they form a $k$-sunflower with core $C$. The core may be empty. The
sets $A_i\setminus C$ are its petals. For $w\ge1$ and $k\ge2$, define
\begin{equation}
 f(w,k)=\max\bigl\{|\F|:\F\text{ is $w$-uniform and contains no
 $k$-sunflower}\bigr\}.
 \label{eq:f-definition}
\end{equation}
The Erd\H{o}s--Rado sunflower lemma makes this maximum finite
\cite{ErdosRado1960}.

\begin{lemma}[Occurrence rule]
\label{lem:occurrence}
Let $A_1,\ldots,A_k$ form a sunflower. Every point in
$A_1\cup\cdots\cup A_k$ belongs to exactly one member or to all $k$ members.
\end{lemma}

\begin{proof}
Suppose that a point belongs to two members. It then belongs to their
intersection, which is the core. Hence it belongs to every member.
\end{proof}

The core of a sunflower of distinct $w$-element sets has size at most $w-1$.
In particular, the possible core sizes for a sunflower of triples are zero,
one, and two.

\subsection{Matchings, degrees, and links}

A matching in a family is a subfamily of pairwise disjoint members. The
matching number $\nu(\F)$ is the largest size of a matching in $\F$. A family
is intersecting if every two distinct members meet. Equivalently,
$\nu(\F)\le1$. Since a matching of size $k$ is a $k$-sunflower with empty
core, every $k$-sunflower-free family satisfies
\begin{equation}
 \nu(\F)\le k-1.
 \label{eq:matching-cap}
\end{equation}
For $m\ge2$, define
\begin{equation}
 \psi(w,k,m)=\max\bigl\{|\F|:\F\text{ is $w$-uniform, contains no
 $k$-sunflower, and }\nu(\F)\le m-1\bigr\}.
 \label{eq:psi-definition}
\end{equation}
This matching-constrained quantity was introduced by Abbott and Hanson
\cite{AbbottHanson1974}.
In particular, $\psi(w,k,2)$ is the maximum over intersecting $w$-uniform
families without a $k$-sunflower, while $\psi(w,k,k)=f(w,k)$.

For a set $S$, define its degree in $\F$ by
\[
 d_{\F}(S)=|\{F\in\F:S\subseteq F\}|.
\]
For $S=\{x\}$, write $d_{\F}(x)$. The link of $S$ is
\[
 \Lk_{\F}(S)=\{F\setminus S:F\in\F,\ S\subseteq F\}.
\]

\begin{lemma}[Link bound]
\label{lem:link-bound}
Let $\F$ be a $w$-uniform family without a $k$-sunflower. If
$|S|=s<w$, then $\Lk_{\F}(S)$ is a $(w-s)$-uniform family without a
$k$-sunflower. Consequently,
\begin{equation}
 d_{\F}(S)=|\Lk_{\F}(S)|\le f(w-s,k).
 \label{eq:link-bound}
\end{equation}
\end{lemma}

\begin{proof}
The members of $\Lk_{\F}(S)$ are distinct. If $k$ of them formed a
sunflower with core $C$, adjoining $S$ would give a sunflower in $\F$ with
core $C\cup S$.
\end{proof}

Two families $\mathcal A$ and $\mathcal B$ are cross-intersecting if
$A\cap B\ne\varnothing$ for every $A\in\mathcal A$ and
$B\in\mathcal B$. A collection of families is pairwise cross-intersecting if
each pair in the collection is cross-intersecting.

\subsection{The graph case}

For integers $D,q\ge1$, let $\CH(D,q)$ be the maximum number of edges in a
graph with maximum degree at most $D$ and matching number at most $q$. The
Chv\'atal--Hanson theorem \cite{ChvatalHanson1976} gives
\begin{equation}
 \CH(D,q)
 =Dq+\left\lfloor\frac D2\right\rfloor
      \left\lfloor\frac{q}{\lceil D/2\rceil}\right\rfloor.
 \label{eq:CH}
\end{equation}

\begin{lemma}[Uniformities one and two]
\label{lem:small-uniformities}
For every $k\ge2$,
\begin{equation}
 f(1,k)=k-1,
 \qquad
 f(2,k)=\CH(k-1,k-1).
 \label{eq:f1-f2}
\end{equation}
\end{lemma}

\begin{proof}
A family of distinct singletons contains a $k$-sunflower exactly when it has
at least $k$ members. This gives the first equality.

View a $2$-uniform family as a graph. A sunflower of distinct edges has empty
core or a one-point core. These two cases are a matching and a set of edges
incident with one vertex. The graph is therefore $k$-sunflower-free exactly
when its matching number and maximum degree are both at most $k-1$.
Equation~\eqref{eq:CH} gives the second equality.
\end{proof}

\subsection{Products}

Abbott proved the following product inequality \cite{Abbott1966}. Frankl and
Wang later proved a vector-valued generalization \cite{FranklWang2025}.

Let $X$ and $Y$ be disjoint. For
$\mathcal A\subseteq\binom Xu$ and $\mathcal B\subseteq\binom Yv$, define
\[
 \mathcal A*\mathcal B
 =\{A\cup B:A\in\mathcal A,\ B\in\mathcal B\}.
\]
This family is $(u+v)$-uniform and has size
$|\mathcal A||\mathcal B|$.

\begin{lemma}[Product construction]
\label{lem:product}
If $\mathcal A$ and $\mathcal B$ contain no $k$-sunflower, then
$\mathcal A*\mathcal B$ contains no $k$-sunflower. Hence
\begin{equation}
 f(u+v,k)\ge f(u,k)f(v,k).
 \label{eq:supermultiplicative}
\end{equation}
\end{lemma}

\begin{proof}
Suppose that $A_i\cup B_i$, for $1\le i\le k$, form a sunflower in the
product family. Intersecting with $X$ and $Y$ shows that the pairwise
intersections among the $A_i$ are equal, and the same holds among the $B_i$.

If two $A_i$ are equal, their common pairwise intersection is that whole set.
It is contained in every other $A_j$. Uniformity then implies that all $A_i$
are equal. Otherwise the $A_i$ are distinct and form a $k$-sunflower in
$\mathcal A$. The latter case is impossible, so all $A_i$ are equal. The
same argument shows that all $B_i$ are equal. The $k$ product members are
then equal, a contradiction.

Taking $\mathcal A$ and $\mathcal B$ to be extremal families on disjoint
supports, the product family has $f(u,k)f(v,k)$ members, which proves
\cref{eq:supermultiplicative}.
\end{proof}

Since $\log f(w,k)$ is superadditive, Fekete's lemma \cite{Fekete1923} and
\cref{eq:supermultiplicative} give the extended limit
\begin{equation}
 B(k)=\lim_{w\to\infty}f(w,k)^{1/w}
     =\sup_{w\ge1}f(w,k)^{1/w}.
 \label{eq:B-limit}
\end{equation}
The value in \cref{eq:B-limit} may be infinite. The sunflower conjecture
states that it is finite for every fixed $k$.

\section{A matching-constrained construction}
\label{sec:constructions}

We first construct larger sunflower-free families from smaller ones while
controlling their matching numbers. We then specialize the resulting
recurrence and combine it with a simple splitting argument to obtain a
recurrence involving only $f(w,k)$. Finally, we solve this recurrence to
derive an explicit lower bound on the exponential growth rate of $f(w,k)$.
An alternative recurrence using a bounded-degree auxiliary graph is recorded
in the appendix.

\subsection{Block-size recurrence}

The construction builds a $w$-uniform sunflower-free family from
smaller families by adjoining either one auxiliary point or a larger
auxiliary block, while preserving the required matching-number bound.

\begin{theorem}[Block-size recurrence]
\label{thm:block-size}
Let $k\ge3$ and $2\le m\le k-1$. Put
\[
 t=k-m+1.
\]
For every $w\ge t+1$,
\begin{equation}
 \psi(w,k,m)
 \ge (k-1)\psi(w-1,k,m)
    +\left\lfloor\frac{k-1}{t}\right\rfloor f(w-t,k).
 \label{eq:block-size}
\end{equation}
\end{theorem}

\begin{proof}
Take pairwise disjoint sets $X,Y,Z$, where $|Y|=k-1$. On $X$, take a
$(w-1)$-uniform family $\A$ with no $k$-sunflower such that
\[
 |\A|=\psi(w-1,k,m),
 \qquad
 \nu(\A)\le m-1.
\]
Let $\mathcal P$ consist of $\lfloor(k-1)/t\rfloor$ pairwise disjoint
$t$-subsets of $Y$. For each $p\in\mathcal P$, take a
$(w-t)$-uniform family $\E_p$ on $Z$ with no $k$-sunflower and with
$|\E_p|=f(w-t,k)$. The families $\E_p$ may have the same support. Define
\begin{equation}
\begin{split}
 \F={}&\{T\cup\{y\}:T\in\A,\ y\in Y\}\\
 &{}\cup\{p\cup S:p\in\mathcal P,\ S\in\E_p\}.
\end{split}
\label{eq:block-family}
\end{equation}
The members in the first and second lines meet $Y$ in one and $t$ points,
respectively. Since $t\ge2$, the two lines are disjoint. All members within
each line are distinct. Consequently,
\[
 |\F|
 =|Y||\A|+\sum_{p\in\mathcal P}|\E_p|
 =(k-1)\psi(w-1,k,m)
   +\left\lfloor\frac{k-1}{t}\right\rfloor f(w-t,k),
\]
which is the right-hand side of \cref{eq:block-size}.

Consider a matching in $\F$ with $q$ members from the first line and $b$
members from the second. If $b=0$, the $X$-parts form a matching in $\A$, so
$q\le m-1$. If $b\ge1$, then, because the matching members are pairwise
disjoint, their intersections with $Y$ are also pairwise disjoint, and
\[
 q+tb\le k-1.
\]
It follows that
\[
 q+b\le k-1-(t-1)b\le k-t=m-1.
\]
Hence $\nu(\F)\le m-1$.

Suppose that $F_1,\ldots,F_k$ form a sunflower in $\F$. Every $F_i$ meets
the $(k-1)$-point set $Y$. By the pigeonhole principle, some $y\in Y$
therefore belongs to at least two of them. By \cref{lem:occurrence}, it
belongs to all of them.

If all $F_i$ come from the first line of \cref{eq:block-family}, write
$F_i=T_i\cup\{y\}$. Deleting $y$ gives a $k$-sunflower of distinct members
of $\A$.

If all $F_i$ come from the second line, every corresponding block in
$\mathcal P$ contains $y$. The blocks in $\mathcal P$ are disjoint, so all
$F_i$ use the same block $p$. Deleting $p$ gives a $k$-sunflower of distinct
members of $\E_p$.

It remains to consider a sunflower using both lines. A member from the first
line and a member from the second intersect exactly in $\{y\}$. The core is
therefore $\{y\}$. The $X$-parts of the members from the first line form a
matching in $\A$, so there are at most $m-1$ such members. All members from
the second line contain $y$. Since the blocks in $\mathcal P$ are disjoint,
they all use the same block $p$. The set $p\setminus\{y\}$ is nonempty
because $t\ge2$, so there is at most one member from the second line. The
sunflower has at most $m\le k-1$ members, a contradiction.
\end{proof}

The case $m=k-1$ gives the recurrence used for the exponential estimate.

\begin{corollary}
\label{cor:endpoint-recurrence}
For every $k\ge3$ and $w\ge3$,
\begin{equation}
 \psi(w,k,k-1)
 \ge (k-1)\psi(w-1,k,k-1)
    +\left\lfloor\frac{k-1}{2}\right\rfloor f(w-2,k).
 \label{eq:endpoint-recurrence}
\end{equation}
\end{corollary}

\subsection{Unrestricted families}

The following inequalities connect the matching-constrained quantities to
$f(w,k)$. The second is a case of the Abbott--Hanson disjoint-support split
\cite{AbbottHanson1974}.

\begin{lemma}[Star and split]
\label{lem:star-split}
For every $k\ge3$ and $w\ge2$,
\begin{align}
 \psi(w,k,2)&\ge f(w-1,k),
 \label{eq:star}\\
 f(w,k)&\ge\psi(w,k,k-1)+\psi(w,k,2).
 \label{eq:split}
\end{align}
\end{lemma}

\begin{proof}
For \cref{eq:star}, take an extremal $(w-1)$-uniform family with
$f(w-1,k)$ members and add one new point to every member. The resulting
family has the same size and is intersecting. Deleting the new point from a
sunflower would give a sunflower in the original family, so the resulting
family is counted by $\psi(w,k,2)$. This proves \cref{eq:star}.

For \cref{eq:split}, take extremal families attaining $\psi(w,k,k-1)$ and
$\psi(w,k,2)$ and place them on disjoint supports. Their union has size
$\psi(w,k,k-1)+\psi(w,k,2)$. A sunflower with nonempty core cannot use
members from both supports. A sunflower with empty core would be a matching
of size $k$, but the matching number of the union is at most
\[
 (k-2)+1=k-1.
\]
Thus the union contains no $k$-sunflower, which proves \cref{eq:split}.
\end{proof}

For $k\ge4$, Frankl and Wang proved
\begin{equation}
 \psi(3,k,2)=f(2,k)
 \label{eq:FW-intersecting-triples}
\end{equation}
\cite[Theorem~1.15]{FranklWang2025}. Thus the star bound
\cref{eq:star} is tight when $w=3$.

\subsection{Exponential rate}

Combining \cref{cor:endpoint-recurrence,lem:star-split} yields simultaneous
recursive lower bounds for $\psi(w,k,k-1)$ and $f(w,k)$.
We use $a_w$ to track lower bounds for $\psi(w,k,k-1)$ and $b_w$ to track
lower bounds for $f(w,k)$.

\begin{proposition}[Recursive lower bounds]
\label{prop:recursive-lower}
Fix $k\ge3$. Let $a_3,b_2,b_3$ be positive integers satisfying
\[
 a_3\le\psi(3,k,k-1),
 \qquad
 b_2\le f(2,k),
 \qquad
 b_3\le f(3,k).
\]
For $w\ge4$, define
\begin{align}
 a_w&=(k-1)a_{w-1}
       +\left\lfloor\frac{k-1}{2}\right\rfloor b_{w-2},
 \label{eq:a-recursion}\\
 b_w&=a_w+b_{w-1}.
 \label{eq:b-recursion}
\end{align}
Then
\[
 a_w\le\psi(w,k,k-1),
 \qquad
 b_w\le f(w,k)
 \qquad(w\ge4).
\]
\end{proposition}

\begin{proof}
We prove both inequalities simultaneously by induction. For $w\ge4$, assume
that
\[
 a_{w-1}\le\psi(w-1,k,k-1),
 \qquad
 b_{w-2}\le f(w-2,k),
 \qquad
 b_{w-1}\le f(w-1,k).
\]
When $w=4$, these are precisely the three assumed initial bounds. By
\cref{eq:endpoint-recurrence},
\begin{align*}
 \psi(w,k,k-1)
 &\ge (k-1)\psi(w-1,k,k-1)
      +\left\lfloor\frac{k-1}{2}\right\rfloor f(w-2,k)\\
 &\ge (k-1)a_{w-1}
      +\left\lfloor\frac{k-1}{2}\right\rfloor b_{w-2}
  =a_w.
\end{align*}
Using \cref{eq:star,eq:split}, we then obtain
\begin{align*}
 f(w,k)
 &\ge \psi(w,k,k-1)+\psi(w,k,2)\\
 &\ge a_w+f(w-1,k)\\
 &\ge a_w+b_{w-1}=b_w.
\end{align*}
These two bounds provide the induction hypotheses at the next index.
\end{proof}

The two recurrences reduce to one recurrence for $b_w$. For $w\ge5$,
\cref{eq:b-recursion} at index $w-1$ gives
\[
 a_{w-1}=b_{w-1}-b_{w-2}.
\]
Substitution into \cref{eq:a-recursion,eq:b-recursion} gives
\begin{equation}
 b_w=k b_{w-1}
     -\left\lceil\frac{k-1}{2}\right\rceil b_{w-2},
 \qquad w\ge5.
 \label{eq:reduced-recursion}
\end{equation}

This recurrence starts at $w=5$ because $b_3$ need not equal $a_3+b_2$.
The construction first gives $a_4$ and then sets $b_4=a_4+b_3$.

\begin{theorem}[Exponential rate]
\label{thm:exponential-rate}
For every $k\ge3$,
\begin{equation}
 B(k)\ge r_k,
 \qquad
 r_k=
 \begin{cases}
 \dfrac{k+\sqrt{k^2-2k}}{2},&k\text{ even},\\[6pt]
 \dfrac{k+\sqrt{k^2-2k+2}}{2},&k\text{ odd}.
 \end{cases}
 \label{eq:exponential-rate}
\end{equation}
\end{theorem}

\begin{proof}
Use the sequences in \cref{prop:recursive-lower}. Consider the polynomial
\[
 x^2-kx+c,
 \qquad
 c=\left\lceil\frac{k-1}{2}\right\rceil.
\]
It is positive at $0$, negative at $1$, and tends to infinity with $x$.
Its roots $s<r$ therefore satisfy $0<s<1<r$. Put
\[
 u_w=b_w-sb_{w-1}.
\]
Equation~\eqref{eq:reduced-recursion} and the identities $r+s=k$ and $rs=c$
give
\begin{align*}
 u_w
 &=b_w-sb_{w-1}\\
 &=(k-s)b_{w-1}-c b_{w-2}\\
 &=r b_{w-1}-rs b_{w-2}\\
 &=r\bigl(b_{w-1}-s b_{w-2}\bigr)
  =r u_{w-1}
 \qquad(w\ge5).
\end{align*}
The initial value is
\[
 u_4=b_4-sb_3=a_4+(1-s)b_3>0.
\]
Therefore
\[
 u_w=u_4r^{w-4}
 \qquad(w\ge4).
\]
Since $b_w=u_w+sb_{w-1}\ge u_w$ and $b_w\le f(w,k)$,
\[
 f(w,k)^{1/w}\ge u_4^{1/w}r^{(w-4)/w}.
\]
Letting $w$ tend to infinity gives $B(k)\ge r$. Substituting the two values
of $c$ into the quadratic formula gives \cref{eq:exponential-rate}.
\end{proof}

\begin{corollary}
\label{cor:rate-values}
The bound in \cref{thm:exponential-rate} gives
\[
 B(4)\ge2+\sqrt2,
 \qquad
 B(5)\ge\frac{5+\sqrt{17}}2,
 \qquad
 B(6)\ge3+\sqrt6.
\]
As $k$ tends to infinity,
\begin{equation}
 r_k=
 \begin{cases}
 k-\dfrac12-\dfrac{1}{4k}+O(k^{-2}),&k\text{ even},\\[6pt]
 k-\dfrac12+\dfrac{1}{4k}+O(k^{-2}),&k\text{ odd}.
 \end{cases}
 \label{eq:rate-expansion}
\end{equation}
\end{corollary}

\begin{proof}
The three displayed values follow by substitution. For the asymptotic
expansion, use
\begin{align*}
 \sqrt{k^2-2k}
   &=k-1-\frac{1}{2k}+O(k^{-2}),\\
 \sqrt{k^2-2k+2}
   &=k-1+\frac{1}{2k}+O(k^{-2}),
\end{align*}
and substitute into \cref{eq:exponential-rate}.
\end{proof}

For comparison, Abbott, Hanson, and Sauer proved
$B(3)\ge\sqrt{10}$ \cite{AbbottHansonSauer1972}. This is stronger than the
$k=3$ specialization of \cref{thm:exponential-rate}. Their stronger
three-petal estimate comes from a substitution construction whose proof is
specific to three petals, whereas \cref{thm:exponential-rate} applies to
every $k\ge3$. The constructions of Abbott and Exoo imply
\[
 B(4)\ge38^{1/3},
 \qquad
 B(5)\ge\sqrt{20},
 \qquad
 B(6)\ge146^{1/3}
\]
\cite{AbbottExoo1992}. The three corresponding bounds in
\cref{cor:rate-values} are larger.

The exponential-rate argument uses \cref{eq:endpoint-recurrence}. The other
matching caps in \cref{thm:block-size} give additional finite recurrences.
Appendix~\ref{sec:capacity-appendix} records the alternative capacity
recurrence.

\section{Upper bounds for triples}
\label{sec:triple-upper}

This section proves an upper bound for $f(3,k)$ that applies for every
$k\ge4$ and has leading term $2k^3$. At $k=5,6,7$ it gives
\[
 f(3,5)\le146,
 \qquad
 f(3,6)\le255,
 \qquad
 f(3,7)\le474.
\]
The argument begins with a largest collection of pairwise disjoint triples.
Every member of the family meets their union, since otherwise the collection
could be enlarged. We classify the members by the number of points they have
in this union. Members meeting it in one point give graphs on the remaining
points; maximality forces certain pairs of these graphs to be
cross-intersecting, while the absence of sunflowers bounds their degrees and
matching numbers. Combining these graph bounds with an incidence count gives
the general upper bound.

The maximum-matching framework is inherited from Abbott and Hanson: they
decompose a family around disjoint triples and analyze the resulting residue
graphs, obtaining the bound $3(k-1)$ in one case and $4(k-1)$ in general
\cite{AbbottHanson1974}. Our graph lemma sharpens the general term to
$3(k-1)$, and the subsequent argument replaces the residue matching cap
$k-1$ by $k-2$. Frankl and Wang give another general upper bound for triple
families \cite[Theorem~1.14]{FranklWang2025}. Both comparison bounds apply
only when $k\ge7$; asymptotically, the Frankl--Wang bound has the smaller
leading term $\frac53k^3$.

\subsection{Cross-intersecting graphs}

We first establish the required bound for three pairwise cross-intersecting
graphs.

\begin{lemma}[Cross-intersecting graph bound]
\label{lem:triple-cross-graphs}
Let $A,B,C$ be simple graphs whose edge families are pairwise
cross-intersecting. Suppose that each graph has maximum degree at most
$k-1$, where $k\ge4$. If at least two of the graphs are nonempty, then
\[
 |E(A)|+|E(B)|+|E(C)|\le3(k-1).
\]
\end{lemma}

\begin{proof}
First suppose that none of $A,B,C$ contains two disjoint edges; equivalently,
each nonempty graph is intersecting. An intersecting graph is a star or a
subgraph of a triangle. Indeed, if it
contains the edges $xy$ and $xz$, then any edge avoiding $x$ must be $yz$.
If such an edge occurs, no further edge is possible outside this triangle.

If one of the graphs contains three edges of a star centered at $x$, every
edge in all three graphs contains $x$. A two-point edge avoiding $x$ cannot
meet three distinct leaves. The degree bound gives at most $k-1$ edges in
each graph. If one graph contains a triangle, every edge in all three graphs
is an edge of that triangle, and the total is at most $9\le3(k-1)$. In the
remaining case each graph has at most two edges.

Now suppose that one graph, say $A$, contains disjoint edges
$P=\{p,q\}$ and $Q=\{r,s\}$. Every edge of $B$ and $C$ has one endpoint in
$P$ and one in $Q$. Thus both are subgraphs of
\[
 K=\{pr,ps,qr,qs\}.
\]

Assume that exactly one of $B,C$ is nonempty, say $B$. The possible edge
sets of $B$, up to symmetry, give the following bounds:
\[
\begin{array}{c|ccccc}
E(B)&\text{one}&\text{two adjacent}&\text{two opposite}&\text{three}&\text{four}\\
\hline
|E(A)|&2(k-1)&k&4&3&2\\
|E(A)|+|E(B)|&2k-1&k+2&6&6&6.
\end{array}
\]
For one prescribed edge, every edge of $A$ belongs to the union of its two
endpoint-stars. For two adjacent prescribed edges, every edge of $A$ belongs
to their common-endpoint star or is the edge joining their other endpoints.
The remaining three bounds follow by listing the two-point sets that meet
every prescribed edge. Each total in the last row is at most $3(k-1)$.

Finally, suppose that both $B$ and $C$ are nonempty. Each edge of $K$ is
disjoint only from its opposite edge. Cross-intersection therefore gives
\begin{equation}
 |E(B)|+|E(C)|\le4.
 \label{eq:triple-BC}
\end{equation}
Put $U=E(B)\cup E(C)$. If $|U|=1$, then $B$ and $C$ consist of the same
single edge. Every edge of $A$ meets this edge, so
\[
 |E(A)|+|E(B)|+|E(C)|
 \le2(k-1)+2=2k\le3(k-1).
\]
If $|U|\ge2$, the possibilities for $U$ are two adjacent edges, two opposite
edges, three edges, or all four edges of $K$. The corresponding bounds for
$|E(A)|$ are $k,4,3,2$. Together with \cref{eq:triple-BC}, the total is at
most
\[
 \max\{k+4,8\}\le3(k-1).
\]
\end{proof}

\subsection{Maximum-matching bound}

\begin{theorem}
\label{thm:triple-general-upper}
For every $k\ge4$, put
\[
 h_k=\CH(k-1,k-1)=f(2,k),
 \qquad
 c_k=\max\{\CH(k-1,k-2),3(k-1)\}.
\]
Then
\begin{equation}
 f(3,k)\le
 \frac{k-1}{2}\bigl(c_k+3h_k-1\bigr).
 \label{eq:triple-general-upper}
\end{equation}
\end{theorem}

\begin{proof}
Let $\F$ be a $3$-uniform family without a $k$-sunflower such that
$|\F|=f(3,k)$. Choose a maximum matching $M_1,\ldots,M_v$ and put
\[
 T=M_1\cup\cdots\cup M_v.
\]
Equation~\eqref{eq:matching-cap} gives $v\le k-1$. Every member of $\F$ meets
$T$, since a member disjoint from $T$ would extend the matching.

For $j\in\{1,2,3\}$, let
\[
 t_j=|\{F\in\F:|F\cap T|=j\}|.
\]
By \cref{lem:link-bound,lem:small-uniformities}, every point has degree at
most $h_k$. Counting incidences between the members of $\F$ and the $3v$
points of $T$ gives
\begin{equation}
 t_1+2t_2+3t_3\le3v h_k.
 \label{eq:triple-incidence}
\end{equation}
The matching members themselves give
\begin{equation}
 t_3\ge v.
 \label{eq:triple-t3}
\end{equation}

Fix $i$ and write $M_i=\{a,b,c\}$. On the points outside $T$, define
\[
 S_a=\{F\setminus\{a\}:F\in\F,\ F\cap T=\{a\}\},
\]
and define $S_b,S_c$ in the same way. These graphs are pairwise
cross-intersecting. If an edge of $S_a$ and an edge of $S_b$ were disjoint,
their corresponding triples, together with the members $M_j$ for $j\ne i$,
would form a matching of size $v+1$.

Each residue graph has maximum degree at most $k-1$. Indeed, if a point $x$
outside $T$ were incident with $k$ edges $\{x,z_1\},\ldots,\{x,z_k\}$ in
$S_a$, then the corresponding triples
$\{a,x,z_1\},\ldots,\{a,x,z_k\}$ would form a $k$-sunflower with core
$\{a,x\}$. Moreover, the link $\Lk_{\F}(a)$ contains the edge $\{b,c\}$,
which is disjoint from every edge of $S_a$. Thus a matching of size $k-1$ in
$S_a$, together with $\{b,c\}$, would form a matching of size $k$ in the link
and hence a $k$-sunflower with core $\{a\}$. Therefore
\[
 \nu(S_a)\le k-2,
 \qquad
 |E(S_a)|\le\CH(k-1,k-2),
\]
and the same inequalities hold for $S_b$ and $S_c$.

If at most one residue graph is nonempty, their combined size is at most
$\CH(k-1,k-2)$. If at least two are nonempty,
\cref{lem:triple-cross-graphs} bounds their combined size by $3(k-1)$.
Consequently, the members counted by $t_1$ that meet a fixed $M_i$ contribute
at most $c_k$. Every member counted by $t_1$ meets a unique matching triple
$M_i$ in exactly one point and is represented by exactly one edge in one of
its three residue graphs. Summing over the $v$ matching triples therefore
gives
\begin{equation}
 t_1\le v c_k.
 \label{eq:triple-t1}
\end{equation}

Eliminating $t_2$ from \cref{eq:triple-incidence} and using
\cref{eq:triple-t3,eq:triple-t1}, we obtain
\begin{align*}
 |\F|
  &=t_1+t_2+t_3\\
  &\le\frac{3v h_k+t_1-t_3}{2}\\
  &\le\frac v2\bigl(c_k+3h_k-1\bigr)\\
  &\le\frac{k-1}{2}\bigl(c_k+3h_k-1\bigr).
\end{align*}
\end{proof}

\subsection{Numerical consequences}

For $k\ge5$, one has
$\CH(k-1,k-2)\ge(k-1)(k-2)\ge3(k-1)$, so $c_k=\CH(k-1,k-2)$.
\Cref{lem:triple-cross-graphs} controls the case of two or three nonempty
residue graphs. At $k=4$, it gives $c_4=9$, and
\cref{thm:triple-general-upper} yields $f(3,4)\le57$.

\begin{corollary}
\label{cor:triple-upper-values}
The general bound gives
\[
 f(3,5)\le146,
 \qquad
 f(3,6)\le255,
 \qquad
 f(3,7)\le474.
\]
\end{corollary}

\begin{proof}
Equation~\eqref{eq:CH} gives
\[
\begin{array}{c|ccc|c}
k&h_k&\CH(k-1,k-2)&c_k&
\dfrac{k-1}{2}(c_k+3h_k-1)\\
\hline
5&20&14&14&146\\
6&27&22&22&255\\
7&42&33&33&474.
\end{array}
\]
Apply \cref{thm:triple-general-upper}.
\end{proof}

\section{Bounds for \texorpdfstring{$f(3,4)$}{f(3,4)}}
\label{sec:f34}

This section proves $39\le f(3,4)\le49$. The lower bound is given by an
explicit family of $39$ triples. For the upper bound, we assume that a
$50$-member family without a $4$-sunflower exists and choose a maximum
matching of three triples. We analyze the local structure around these
triples, decompose the family into rooted blocks and cross classes, and reduce
the possible configurations to four numerical profiles. Finite exhaustive
searches then rule out all four profiles.

\subsection{A family of 39 triples}

\begin{theorem}
\label{thm:f34-lower}
One has
\[
 f(3,4)\ge39.
\]
\end{theorem}

\begin{proof}
Let $V_1,V_2,V_3$ be pairwise disjoint sets of size $3$. Take four further
points $a,\ell_{12},\ell_{13},\ell_{23}$. Define
\begin{align*}
 \mathcal G&=\{V_1,V_2,V_3\},\\
 \mathcal P&=\bigl\{\{a\}\cup P:
        P\in\tbinom{V_i}{2},\ 1\le i\le3\bigr\},\\
 \mathcal Q&=\bigl\{\{\ell_{ij},x,y\}:
        1\le i<j\le3,\ x\in V_i,\ y\in V_j\bigr\}.
\end{align*}
Put $\F_{39}=\mathcal G\cup\mathcal P\cup\mathcal Q$. These three families
are disjoint and
\[
 |\F_{39}|=3+3\binom32+3\cdot3^2=39.
\]

We exclude the three possible core sizes. Every pair has degree at most $3$.
The pairs $\{\ell_{ij},x\}$ with $x\in V_i\cup V_j$ have degree $3$; all
other pair types have degree at most $2$. Thus there is no $4$-sunflower
with a two-point core.

The link at $a$ is the disjoint union of three triangles. The link at
$\ell_{ij}$ is the complete bipartite graph between $V_i$ and $V_j$. For
$x\in V_i$, its link is the disjoint union of a triangle on
$\{a\}\cup(V_i\setminus\{x\})$ and two three-edge stars. The stars have
centres $\ell_{ij}$ and leaves $V_j$, for the two indices $j\ne i$. Each of
these link graphs has matching number $3$. Hence there is no $4$-sunflower
with a one-point core.

It remains to exclude four disjoint members. Let $g$ be the number of group
triples $V_i$ in a matching. If $g=3$, no further member is disjoint from all
three groups. If $g=2$, every remaining member uses points of the third
group. The only possibilities are the third group triple and the three
members of $\mathcal P$ supported on that group. These four members are
pairwise intersecting, so at most one can be added.

Suppose that $g=1$, say $V_1$ is used. Among members avoiding $V_1$, the only
special points are $a$ and $\ell_{23}$, apart from the two group triples
$V_2,V_3$. If one of those group triples is used, at most one further member
can be disjoint from it. If neither is used, at most two members can be
chosen because every available member contains $a$ or $\ell_{23}$. Thus the
matching has at most three members.

Finally suppose that $g=0$ and that four non-group members are disjoint.
They must use the four special points
$a,\ell_{12},\ell_{13},\ell_{23}$ exactly once. The member containing $a$
uses two points of some group, say $V_1$. The members containing
$\ell_{12}$ and $\ell_{13}$ use two further distinct points of $V_1$. This
would require four points in $V_1$, a contradiction. Hence
$\nu(\F_{39})\le3$, so no $4$-sunflower has empty core.
\end{proof}

\subsection{Local structure around a triple}

The upper bound follows from finite enumeration. The next lemmas give the
mathematical reduction and define the objects that are enumerated.

\begin{lemma}[Degree bounds]
\label{lem:f34-degrees}
Let $\F$ be a $3$-uniform family without a $4$-sunflower. Every pair has
degree at most $3$, and every point has degree at most $10$.
\end{lemma}

\begin{proof}
Four distinct triples containing the same pair form a sunflower with that pair
as core. For a point $x$, the link graph $\Lk_{\F}(x)$ has maximum degree and
matching number at most $3$. Equation~\eqref{eq:CH} gives
\[
 d_{\F}(x)=|E(\Lk_{\F}(x))|\le\CH(3,3)=10.
\]
\end{proof}

Fix $S=\{s_1,s_2,s_3\}\in\F$. For $s\in S$, define the exact-singleton
petal graph
\[
 L'_s(S)=\{F\setminus\{s\}:F\in\F,\ F\cap S=\{s\}\}.
\]

\begin{lemma}[Exact-singleton links]
\label{lem:f34-singlelink}
For every $s\in S$,
\[
 \Delta(L'_s(S))\le3,
 \qquad
 \nu(L'_s(S))\le2,
 \qquad
 |E(L'_s(S))|\le7.
\]
\end{lemma}

\begin{proof}
The degree bound follows from \cref{lem:f34-degrees}. Three disjoint edges of
$L'_s(S)$, together with the edge $S\setminus\{s\}$ in $\Lk_{\F}(s)$, would
give a four-matching in this link. Thus $\nu(L'_s(S))\le2$, and
\cref{eq:CH} gives $\CH(3,2)=7$.
\end{proof}

In the maximum-matching decomposition used below, the members associated
exclusively with one matching triple form an intersecting family containing
that triple; the following definition records this structure.

A rooted block is an intersecting, $3$-uniform family $\A$ without a
$4$-sunflower, together with a distinguished member $T\in\A$. Put
\[
 p(\A,T)=|\{A\in\A\setminus\{T\}:|A\cap T|=2\}|.
\]

\begin{lemma}[Rooted blocks]
\label{lem:f34-block}
Every intersecting, $3$-uniform family without a $4$-sunflower has at most
$10$ members \cite[Proposition~6.3]{FranklWang2025}. The rooted
isomorphism-class counts used below are
\begin{center}
\begin{tabular}{@{}lr@{}}
\toprule
rooted block list & classes\\
\midrule
$|\A|=8$ and $p(\A,T)=0$ & $45$\\
$|\A|=9$ and $p(\A,T)=0$ & $27$\\
all rooted blocks with $|\A|=9$ & $665$\\
all rooted blocks with $|\A|=10$ & $114$\\
\bottomrule
\end{tabular}
\end{center}
\end{lemma}

\begin{proof}[Computer-assisted proof]
The cited result gives the size bound. The enumeration below independently
recovers it and establishes the rooted class counts.
Relabel the anchor as $\{0,1,2\}$. Adjoin the other members in lexicographic
order. Outside labels obey the first-appearance rule: a new outside label is
the least unused label. This preserves one representative of every rooted
isomorphism type. To see this, choose a labelling with least sorted member
list. If a larger label appeared before a smaller unused label, interchanging
the two would decrease the list.

A block with $m$ members uses at most $2(m-1)$ outside points. The
enumeration at sizes $8,9,10,11$ therefore needs at most $14,16,18,20$
outside labels, respectively. The nodes of the enumeration tree are partial
rooted blocks. A node is rejected exactly when it contains a repeated
triple, a disjoint pair, or a $4$-sunflower. Canonicalization
permits every anchor permutation and every permutation of the outside
support, and gives the four class counts in the statement. No valid node
reaches size $11$, while valid size-$10$ blocks exist.
\end{proof}

Define the closed meeting neighborhood of $S$ by
\[
 \mathcal N_{\F}[S]=\{F\in\F:F\cap S\ne\varnothing\}.
\]
For $e\in\binom S2$, let $n_e$ be the number of members other than $S$
containing $e$. Put
\[
 s=\sum_{e\in\binom S2}n_e,
 \qquad
 \mu_x=\sum_{\substack{e\in\binom S2\\x\in e}}n_e
 \quad(x\in S).
\]
By \cref{lem:f34-degrees,lem:f34-singlelink},
\[
 0\le n_e\le2,
 \qquad
 |E(L'_x(S))|\le\min\{7,9-\mu_x\}.
\]
Each member meeting $S$ in two points contributes to $\mu_x$ at both of
those points, so
\[
 \sum_{x\in S}\mu_x=2s.
\]
Consequently,
\begin{equation}
\begin{aligned}
 |\mathcal N_{\F}[S]|
  &=1+s+\sum_{x\in S}|E(L'_x(S))|\\
  &\le1+s+\min\{21,27-2s\}\le25.
\end{aligned}
\label{eq:f34-neighborhood25}
\end{equation}

\begin{lemma}[Neighborhood equality]
\label{lem:f34-rigidity}
Equality in \cref{eq:f34-neighborhood25} is possible only when
\[
 n_e=1\quad(e\in\tbinom S2),
 \qquad
 |E(L'_x(S))|=7,
 \qquad
 d_{\F}(x)=10\quad(x\in S).
\]
The three members meeting $S$ in two points then have a common third point
$y$, and $y$ is outside every graph $L'_x(S)$.
\end{lemma}

\begin{proof}
For integral $0\le s\le6$, the last expression in
\cref{eq:f34-neighborhood25} is uniquely maximized at $s=3$. Equality forces
all three singleton links to have seven edges and all $\mu_x$ to equal $2$.
The three equations for the $n_e$ give $n_e=1$ for every internal pair.

Up to isomorphism, the unique seven-edge graph with maximum degree at most
$3$ and matching number at most $2$ is
\[
 R=K_5-(P_3\mathbin{\dot\cup}K_2).
\]
Its degree sequence is $(3,3,3,3,2)$, and $R-v$ has a two-matching for every
vertex $v$. This follows by listing the edges incident with the four
endpoints of a maximum two-matching and applying the degree and matching
bounds.

Write the three two-point members as
\[
 \{s_2,s_3,y_1\},\qquad
 \{s_1,s_3,y_2\},\qquad
 \{s_1,s_2,y_3\}.
\]
If $y_2\ne y_3$, the two corresponding edges in $\Lk_{\F}(s_1)$ are
disjoint. Deleting $y_2,y_3$ from $L'_{s_1}(S)\cong R$ must then destroy
every two-matching. If, say, $y_2$ were not a vertex of $R$, deleting
$y_2,y_3$ would remove at most one vertex of $R$, and the remaining graph
would still contain a two-matching. Hence both $y_2$ and $y_3$ belong to
$R$. Since each already occurs with $s_1$ in one of the two displayed
members, the pair-degree bound gives degree at most two in $R$. They must
therefore both be the unique degree-two vertex of $R$, a contradiction.
Hence $y_2=y_3$, and symmetry gives $y_1=y_2=y_3=:y$.

The pair $\{s_1,y\}$ already occurs in two displayed members. If $y$ were a
vertex of $L'_{s_1}(S)$, its degree there would be at most $1$, whereas $R$
has minimum degree $2$. Thus $y\notin V(L'_{s_1}(S))$, and the same holds at
the other two points of $S$.
\end{proof}

\subsection{The maximum-matching decomposition}

\begin{lemma}
\label{lem:f34-nu2}
If $\F$ is $3$-uniform, has no $4$-sunflower, and $\nu(\F)\le2$, then
$|\F|\le35$. Equivalently,
\[
 \psi(3,4,3)\le35.
\]
\end{lemma}

\begin{proof}
The case $\nu(\F)\le1$ follows from \cref{lem:f34-block}. Otherwise fix a
maximum matching $T_1,T_2$. Apart from the anchors, classify a member by its
intersection sizes with $(T_1,T_2)$. Let $m,p,q,r$ count, in both
orientations, the types
\[
 (2,1),\qquad(2,0),\qquad(1,1),\qquad(1,0),
\]
respectively. These types exhaust the family, so
\[
 |\F|=2+m+p+q+r.
\]
Summing degrees over the six anchor points gives
\begin{equation}
 3m+2p+2q+r\le54.
 \label{eq:f34-nu2-degree}
\end{equation}
Indeed, the six points have total degree at most $60$, while the two anchors
themselves contribute six incidences. Among the other members, the four
types contribute $3,2,2,1$ anchor incidences, respectively.
Summing the exact-singleton link bound over the same points gives
\begin{equation}
 m+2q+r\le42.
 \label{eq:f34-nu2-links}
\end{equation}
There are six such links, each with at most seven edges; a member of type
$m,q,r$ is counted once, twice, or once, respectively, while a member of type
$p$ is not counted.
For $i=1,2$, the members missing the other anchor form an intersecting rooted
block. Applying \cref{lem:f34-block} to both blocks gives
\begin{equation}
 p+r\le18.
 \label{eq:f34-nu2-blocks}
\end{equation}
Each block has at most ten members and includes its anchor, so the two blocks
contain at most eighteen non-anchor members in total.
One quarter of each of the first two inequalities and one half of the third
give
\[
 m+p+q+r\le\frac{54}{4}+\frac{42}{4}+\frac{18}{2}=33.
\]
Thus $|\F|\le35$.
\end{proof}

Assume for a contradiction that $\F$ has $N=50$ members and no
$4$-sunflower. A matching of four triples would be a $4$-sunflower with
empty core, so $\nu(\F)\le3$. By \cref{lem:f34-nu2}, the inequality
$\nu(\F)\le2$ would imply $|\F|\le35$. Thus $\nu(\F)=3$. Fix a maximum
matching $T_1,T_2,T_3$. For $\{i,j,k\}=\{1,2,3\}$, define
\begin{align*}
 \F_i&=\{F\in\F:F\cap T_i\ne\varnothing,
              \ F\cap T_j=F\cap T_k=\varnothing\},&a_i&=|\F_i|,\\
 \C_i&=\{F\in\F:F\cap T_i=\varnothing,
              \ F\cap T_j\ne\varnothing,
              \ F\cap T_k\ne\varnothing\},&c_i&=|\C_i|,\\
 \mathcal Z&=\{F\in\F:F\cap T_1,F\cap T_2,F\cap T_3
              \text{ are all nonempty}\},&z&=|\mathcal Z|.
\end{align*}
These classes partition $\F$ and
\begin{equation}
 N=\sum_{i=1}^3a_i+\sum_{i=1}^3c_i+z.
 \label{eq:f34-partition}
\end{equation}
Each $\F_i$ is an intersecting rooted block: two disjoint members, together
with $T_j,T_k$, would form a four-matching. Hence
\begin{equation}
 1\le a_i\le10.
 \label{eq:f34-a10}
\end{equation}

Let $\D_i$ be the family of members disjoint from $T_i$, and put
\[
 D_i=|\D_i|,
 \qquad
 N_i=|\mathcal N_{\F}[T_i]|,
 \qquad
 V_i=N-2+a_i-D_i.
\]
This choice connects the local neighborhood and block bounds to
the global identity~\eqref{eq:f34-sumV} below, which restricts the possible
$50$-member profiles. The definitions give
\begin{equation}
 D_i=a_j+a_k+c_i,
 \qquad
 N_i=N-D_i,
 \qquad
 V_i=N_i+a_i-2.
 \label{eq:f34-derived}
\end{equation}
Each $\F_i$ is counted in two of the families $\D_1,\D_2,\D_3$, each
$\C_i$ is counted in one, and $\mathcal Z$ is counted in none. Therefore
\[
 \sum_{i=1}^3D_i=2\sum_{i=1}^3a_i+\sum_{i=1}^3c_i.
\]
Combining this identity with \cref{eq:f34-partition,eq:f34-derived} gives
\begin{equation}
 \sum_{i=1}^3V_i=2N-6+z.
 \label{eq:f34-sumV}
\end{equation}
The family $\D_i$ has matching number at most $2$. Therefore
\cref{lem:f34-nu2,eq:f34-neighborhood25,eq:f34-a10,eq:f34-derived} give
\begin{equation}
 D_i\le35,
 \qquad
 N_i\le25,
 \qquad
 V_i\le33.
 \label{eq:f34-basic-profile}
\end{equation}

\subsection{Rigid blocks and finite profiles}

Call $\F_i$ rigid when $N_i=25$. By \cref{lem:f34-rigidity}, the three
members meeting $T_i$ in two points then have a common third point $y$.

\begin{lemma}[Rigid-block interaction]
\label{lem:f34-interaction}
Let $\F_i$ be rigid.
\begin{enumerate}[label=\textup{(\roman*)}]
\item If $y$ lies outside $T_1\cup T_2\cup T_3$, then $a_i=4$.
\item If $a_i\ge5$, then $y$ lies in another anchor, say $T_j$. No
      non-anchor member of $\F_i$ meets $T_i$ in two points. Thus
      $p(\F_i,T_i)=0$, and the $a_i-1$ non-anchor members meet $T_i$ in
      exactly one point.
\item Suppose that $y\in T_j$, put
\[
 \B_y=\{F\in\F_j:y\notin F\},
 \qquad r=a_i-1.
\]
For the three values needed below,
\[
\begin{array}{c|rrr}
r&7&8&9\\ \hline
|\B_y|&\le3&\le2&\le1.
\end{array}
\]
If $\F_j$ is not rigid, then $d_{\F_j}(y)\le7$. When $r=8$, a non-rigid
$\F_j$ cannot have size $9$. If $\F_j$ is rigid, then either $a_j=4$ or
$d_{\F_j}(y)\le2$.
\end{enumerate}
\end{lemma}

\begin{proof}[Proof (partly computer-assisted)]
Suppose first that $y$ is outside the three anchors. A member of $\F_i$
meeting $T_i$ in one point must intersect the two-point member supported on
the other two points of $T_i$. It can do so only through $y$, contrary to
\cref{lem:f34-rigidity}. The anchor and the three two-point members are
therefore all of $\F_i$. This proves (i). If $a_i\ge5$, part~(i) and
$y\notin T_i$ place $y$ in another anchor. The three two-point members then
lie outside $\F_i$, which proves (ii).

For (iii), delete the anchor point from every one-point member of $\F_i$.
This gives three edge families $R_1,R_2,R_3$, indexed by the points of
$T_i$. They are pairwise cross-intersecting, each has maximum degree at most
$3$ and matching number at most $2$, and their total size is $r$.

Every member of $\B_y$ must cover each edge in $R_1\cup R_2\cup R_3$ with
its one- or two-point part outside the anchors. Indeed, let
$A=\{s\}\cup e\in\F_i$ correspond to an edge $e\in R_s$. If
$B\in\B_y$ avoided $e$, then $A,B,T_k$, and the rigid two-point member
$(T_i\setminus\{s\})\cup\{y\}$ would be four disjoint triples, where
$\{i,j,k\}=\{1,2,3\}$.

It remains to justify the finite cover bounds. Let $R_1$ be largest, so
$|E(R_1)|\ge3$. If $R_1$ is a three-edge star, cross-intersection forces the
whole union to be a star. Every one- or two-point cover contains its centre.
All $r$ one-point members of $\F_i$ and all members of $\B_y$ then contain
that point, so the point-degree bound gives $|\B_y|\le10-r$.

If $R_1$ is a triangle, every cross-intersecting edge lies on its three
vertices. Otherwise $R_1$ has a two-matching, and every edge in the other two
graphs lies in $V(R_1)$. By definition, $V(R_1)$ contains only vertices
incident with an edge. If $c$ is the number of connected components of
$R_1$, then $|V(R_1)|\le |E(R_1)|+c$. Choosing one edge from each component
gives a matching, so $c\le\nu(R_1)$. Therefore, for $|E(R_1)|\le6$,
\[
 |V(R_1)|\le |E(R_1)|+\nu(R_1)\le8.
\]
The unique seven-edge graph from \cref{lem:f34-rigidity} has five vertices.
Thus every non-star case has a representative on eight vertices. The cover
enumerator exhausts the three edge sets and their one- and two-point covers on
these eight vertices. It applies the pair-degree and point-degree bounds and
the necessary matching tests in the links of the outside points. The maxima
are $3,2,1$ for $r=7,8,9$. The calculation uses only necessary conditions,
so these are upper bounds for the original family.

To bound $d_{\F_j}(y)$, observe that the three rigid two-point members give a
triangle in $L'_y(T_j)$. The one-point members of $\F_j$ through $y$ have petals
disjoint from this triangle. Let $p_y$ and $r_y$ be the numbers of two-point
and one-point members of $\F_j$ through $y$. The degree and matching bounds
give
\[
 p_y\le\mu_y,
 \qquad
 |E(L'_y(T_j))|\ge3+r_y.
\]
Since $1+\mu_y+|E(L'_y(T_j))|\le10$, we obtain
$p_y+r_y\le6$ and therefore
\[
 d_{\F_j}(y)=1+p_y+r_y\le7.
\]

Suppose that $r=8$, the family $\F_j$ is non-rigid, and $a_j=9$. The cover and
degree bounds force
\[
 |\B_y|=2,
 \qquad
 p_y+r_y=6,
 \qquad
 p_y=\mu_y,
 \qquad
 |E(L'_y(T_j))|=3+r_y.
\]
The link at $y$ already contains the two disjoint edges
$T_j\setminus\{y\}$ and any edge of the rigid triangle on $T_i$. Every
one-point petal in $\F_j$ avoids both $T_i$ and
$T_j\setminus\{y\}$. Two disjoint such petals would therefore complete a
four-matching in the link at $y$. Thus the $r_y$ one-point petals in
$\F_j$ are intersecting, so $r_y\le3$, while $\mu_y\le4$. Hence
$(p_y,r_y)$ is $(3,3)$ or $(4,2)$.

Write $T_j=\{y,b,b'\}$. For $B\in\B_y$, call $B\setminus T_j$ its outside
part. Each outside part has size one or two and covers
$R_1\cup R_2\cup R_3$. If both have size two, they are disjoint, equal, or
intersect in exactly one point. Together with the singleton case, these are
the four cases treated below.

\emph{Case 1: one outside part is a singleton $\{x\}$.}  All eight petals
from $\F_i$ contain $x$.
Intersectingness of $\F_j$ also puts $x$ in all $r_y\ge2$ one-point petals
through $y$. The member of $\B_y$ contains $x$ as well, so
$d_{\F}(x)\ge11$.

\emph{Case 2: the outside parts are disjoint pairs $C,D$.}  If the two
members of $\B_y$ use different anchor points, then the two full triples are
disjoint, contrary to the intersectingness of $\F_j$. Suppose instead that they use
the same anchor point, say $b$. A two-point member of $\F_j$ through $y$
that contains $b'$ would need its third point to belong to both $C$ and $D$
in order to meet both members of $\B_y$, which is impossible. Thus every
two-point member through $y$ contains $b$. The pair $\{y,b\}$ already
belongs to $T_j$ and has degree at most three, so $p_y\le2$, contrary to
$p_y\ge3$.

\emph{Case 3: the outside parts are the same pair $\{x,x'\}$.}  The two
members of $\B_y$ then use different anchor points.
Pair degree leaves at most one petal from $\F_i$ equal to $\{x,x'\}$. The
remaining petals contain $x$ or $x'$. Call one an $x$-petal if it contains
$x$ but not $x'$, and define an $x'$-petal symmetrically. If there are at
most three of each type, there are at most seven petals in total. Thus one
type occurs at least four times; suppose that there are at least four
$x$-petals.

If an $x'$-petal $\{x',v\}$ occurs in $R_s$, cross-intersection forces every
$x$-petal in either of the other two residue graphs to equal $\{x,v\}$, so
there is at most one in each. Since $\Delta(R_s)\le3$ and there are at least
four $x$-petals altogether, $R_s$ contains at least two distinct
$x$-petals. These two edges prevent an $x'$-petal from occurring in either
other residue graph. Moreover, at least one $x$-petal occurs outside $R_s$;
it is $\{x,v\}$ and forces every $x'$-petal in $R_s$ to equal
$\{x',v\}$. Hence there are at most five $x$-petals, one $x'$-petal, and
one petal $\{x,x'\}$, again at most seven in total. Therefore no
$x'$-petal occurs, and all eight petals contain $x$. The symmetric argument
applies if the $x'$-petals are the type occurring at least four times.

Thus all eight one-point members of $\F_i$ and both members of $\B_y$
contain one common point, say $x$. Each of the six members counted by
$p_y+r_y$ must contain $x$ or $x'$ in order to meet both members of
$\B_y$. At most three contain $x'$, by the pair-degree bound for
$\{y,x'\}$. Hence at least one additional member contains $x$, giving
degree at least $11$.

\emph{Case 4: the outside parts share exactly one point, say $\{w,c\}$ and
$\{w,d\}$.}  Every petal from
$\F_i$ contains $w$ or equals $\{c,d\}$. At most two of the $p_y$ two-point
members of $\F_j$ and at most one of its $r_y$ one-point members avoid $w$.
Thus at least three of the six members counted by $p_y+r_y$ contain $w$.

If none of $R_1,R_2,R_3$ contains $\{c,d\}$, all eight petals from $\F_i$
contain $w$. Together with the two members of $\B_y$ and at least three
further members of $\F_j$, this gives $d_{\F}(w)\ge13$.

If at least two residue graphs contain $\{c,d\}$, cross-intersection forces
every petal to be one of $\{c,d\}$, $\{w,c\}$, and $\{w,d\}$. The first
pair can occur in at most three members of $\F_i$. Each of the other two
pairs already occurs in one member of $\B_y$, so each can occur in at most
two members of $\F_i$. Thus $\F_i$ has at most $3+2+2=7$ one-point
members, contrary to $r=8$.

Finally, suppose that exactly one residue graph contains $\{c,d\}$. Each
of the other two graphs can then contain only $\{w,c\}$ and $\{w,d\}$, so
has at most two edges. Reaching eight petals forces both graphs to contain
both edges, while the graph containing $\{c,d\}$ contains three further
edges through $w$. Hence seven members of $\F_i$ contain $w$. Together
with the two members of $\B_y$ and at least three further members of $\F_j$,
this gives $d_{\F}(w)\ge12$. This excludes a non-rigid size-$9$ family
$\F_j$.

Suppose finally that $\F_j$ is rigid. If $a_j=4$, the assertion holds.
Otherwise (ii), applied to $\F_j$, gives $p(\F_j,T_j)=0$.
The graph $L'_y(T_j)$ is the seven-edge graph $R$ from
\cref{lem:f34-rigidity}. Indeed, $R$ is $K_{2,3}$ together with one edge
inside its three-vertex part. Every triangle contains that additional edge,
and the two vertices outside the triangle form an edge. Thus every triangle
in $R$ leaves exactly one disjoint edge. At most one one-point petal of
$\F_j$ avoids the rigid triangle, and $d_{\F_j}(y)\le2$.
\end{proof}

The profile variables satisfy the aggregate bound
\begin{equation}
 \sum_{i=1}^3V_i\le96.
 \label{eq:f34-aggregateV}
\end{equation}
If no $V_i$ equals $33$, this follows from integrality and
\cref{eq:f34-basic-profile}. If $V_i=33$, then
$(N_i,a_i)=(25,10)$. This rigid $10$-block has $r=9$ and common point in
another anchor $T_j$. The cover table gives $|\B_y|\le1$. Since the
members of $\F_j$ split according to whether they contain $y$,
\[
 a_j=d_{\F_j}(y)+|\B_y|.
\]
If $\F_j$ is rigid, \cref{lem:f34-interaction} gives either $a_j=4$ or
$d_{\F_j}(y)\le2$. Thus $a_j\le4$ and $V_j\le27$. If $\F_j$ is
non-rigid, the same lemma gives $d_{\F_j}(y)\le7$, so $a_j\le8$; moreover,
$N_j\le24$, and hence $V_j\le30$. Together with $V_k\le33$, both cases
imply \cref{eq:f34-aggregateV}.

We now encode the local structure around each anchor by finitely many integer
variables for use in the profile enumeration. Write an anchor as
$T_j=\{v_0,v_1,v_2\}$. Let
\begin{itemize}
\item $n_{uv}$ be the number of members meeting $T_j$ exactly in
      $\{u,v\}$, and let $q_{uv}$ count those members lying in $\F_j$;
\item $L_v=|E(L'_v(T_j))|$, and let $r_v$ count the corresponding one-point
      members lying in $\F_j$;
\item $k_v$ be the number of rigid blocks whose common point is $v$.
\end{itemize}
The profile verifier exhausts
\begin{align}
 &0\le q_{uv}\le n_{uv}\le2,
 \qquad 0\le r_v\le L_v\le7,
 \label{eq:f34-local-ranges}\\
 &1+\sum_{\{u,v\}\in\binom{T_j}{2}}n_{uv}
      +\sum_{v\in T_j}L_v=N_j,
 \qquad
 1+\sum_{\{u,v\}\in\binom{T_j}{2}}q_{uv}
      +\sum_{v\in T_j}r_v=a_j,
 \label{eq:f34-local-identities}\\
 &1+\sum_{u\in T_j\setminus\{v\}}n_{uv}+L_v\le10,
 \qquad
 L_v\ge r_v+3k_v.
 \label{eq:f34-local-degree}
\end{align}
If $k_v\ge2$, the two induced triangles are vertex-disjoint and exhaust the
link: any additional edge, together with one edge from each triangle chosen
to avoid its endpoints, would form a three-matching. Hence $L_v=6$ and
$r_v=0$. If $k_v\ge1$, the petals counted by $r_v$ avoid the induced
triangle. Two disjoint such petals, together with an edge of the triangle,
would again form a three-matching. They are therefore intersecting, and the
degree bound gives $r_v\le3$. When $L_v=7$, the link is the graph $R$ from
\cref{lem:f34-rigidity}, in which a triangle has only one disjoint edge; thus
$r_v\le1$. The number of members of $\F_j$ avoiding $v$ is bounded by the
applicable entry of \cref{lem:f34-interaction}. For a rigid family $\F_j$,
\cref{lem:f34-rigidity} forces $n_{uv}=1$, $L_v=7$, and $q_{uv}=0$ unless
$a_j=4$.

For each $v\in T_j$, form the pair
\[
 \left(\sum_{u\in T_j\setminus\{v\}}q_{uv},r_v\right).
\]
The first coordinate counts the non-anchor members of $\F_j$ that meet
$T_j$ in two points and contain $v$; the second counts those meeting $T_j$
exactly at $v$. We call the resulting triple of pairs the anchor signature.
It must be the signature of an actual rooted block of size $a_j$. The
complete block enumeration gives
\[
 2,6,13,23,35,49,51,41,16
\]
signatures at sizes $2$ through $10$, respectively. All ranges used by the
profile calculation are therefore finite.

\begin{proposition}[Profile reduction]
\label{prop:f34-profiles}
Up to a simultaneous permutation of the three anchors, a hypothetical
$50$-member family has one of the following four profiles.
\begin{center}
\begin{tabular}{@{}ccccc@{}}
\toprule
&$(a_1,a_2,a_3)$&$(c_1,c_2,c_3)$&$z$&$(N_1,N_2,N_3)$\\
\midrule
I   &$(8,9,10)$&$(6,8,9)$&$0$&$(25,24,24)$\\
II  &$(8,9,10)$&$(7,7,9)$&$0$&$(24,25,24)$\\
III &$(9,9,10)$&$(7,7,8)$&$0$&$(24,24,24)$\\
IV  &$(8,9,10)$&$(7,8,8)$&$0$&$(24,24,25)$\\
\bottomrule
\end{tabular}
\end{center}
Their $V$-vectors are, respectively,
\[
 (31,31,32),\qquad
 (30,32,32),\qquad
 (31,31,32),\qquad
 (30,31,33).
\]
\end{proposition}

\begin{proof}[Computer-assisted proof]
Order the anchors so that $a_1\le a_2\le a_3$. Enumerate the integers
$1\le a_i\le10$ and nonnegative $c_1,c_2,c_3,z$. Test the partition and
derived identities in
\cref{eq:f34-partition,eq:f34-derived}. Apply
\cref{eq:f34-basic-profile}, the neighborhood and aggregate bounds, and the
fact proved in \cref{lem:f34-cross-bounds} below that the cross class between
two rooted blocks of size $10$ has at most two members. This leaves $55$
numerical profiles. The conditions in \cref{lem:f34-interaction} leave $11$.
Next impose the local ranges in \cref{eq:f34-local-ranges} and the identities
in \cref{eq:f34-local-identities}. The degree conditions in
\cref{eq:f34-local-degree} and the complete anchor signatures leave exactly
the four displayed profiles.

A separate verifier implements the same constraints, reproduces the counts
$55$, $11$, and $4$, and checks the four displayed profiles against the
cross-family bounds below.
\end{proof}

\subsection{Cross-family bounds}

To eliminate the four remaining profiles, we bound the number of members
meeting two anchors in terms of the sizes and structures of the rooted blocks
attached to them. This subsection is independent of the standing
$50$-member assumption: its bounds concern an arbitrary $3$-uniform family
with matching number two and no $4$-sunflower. \Cref{thm:f34-upper}
applies them to the cross classes $\C_i$.

Let $U_1,U_2$ be a maximum two-matching in a $3$-uniform family $\mathcal H$
without a $4$-sunflower, where $\nu(\mathcal H)=2$. Define
\begin{align*}
 \A_1&=\{H\in\mathcal H:H\cap U_1\ne\varnothing,
                         \ H\cap U_2=\varnothing\},\\
 \A_2&=\{H\in\mathcal H:H\cap U_2\ne\varnothing,
                         \ H\cap U_1=\varnothing\},\\
 \X&=\{H\in\mathcal H:H\cap U_1\ne\varnothing,
                       \ H\cap U_2\ne\varnothing\}.
\end{align*}
The families $\A_1,\A_2$ are rooted blocks. Let $\mc(a,b)$ be the maximum
of $|\X|$ under the conditions $|\A_1|\ge a$ and $|\A_2|\ge b$. Let
$\mcz(a,b)$ be the same maximum with $p(\A_1,U_1)=0$.

Deleting a non-anchor block member preserves the two anchors, all cross
members, the absence of a $4$-sunflower, and matching number $2$. These
quantities are therefore non-increasing in each block-size argument. Since
rooted blocks have size at most $10$, the exact lists at sizes $8,9,10$
suffice for the values below.

\begin{lemma}[Exact cross-family bounds]
\label{lem:f34-cross-bounds}
One has
\begin{equation}
 \mcz(8,10)\le7,
 \qquad
 \mcz(9,10)\le6,
 \qquad
 \mc(10,9)\le6,
 \qquad
 \mc(10,10)\le2.
 \label{eq:f34-cross-bounds}
\end{equation}
\end{lemma}

\begin{proof}[Computer-assisted proof]
The $10$-by-$10$ calculation checks all $6555$ unordered pairs of the $114$
rooted $10$-block classes. No pair admits a cross class of size $3$.

For the other bounds, use the complete rooted lists in
\cref{lem:f34-block}. Make the two anchors disjoint. Every outside point of
the second block is either identified injectively with an unused outside
point of the first block or assigned the next fresh label. This enumerates
every possible intersection of the outside supports, since such an
identification is a partial bijection.

A partial identification is rejected only if a necessary condition fails:
distinctness, point degree at most $10$, pair degree at most $3$, or an
already completed $4$-sunflower. After a full identification, the program
generates every possible cross triple. It contains a point of each anchor;
its third point is in the merged block support or is fresh. A target of $t$
cross members needs at most $t$ fresh labels, since each cross triple has at
most one point outside the merged support. Fresh labels are interchangeable.

The subset recursion tests distinctness, every three-matching, and every
four-subfamily for the sunflower condition. Its symmetry reduction orders
the images of interchangeable degree-one outside twins with the same other
pair. These twins are freely permutable, so one representative of each orbit
remains. The recursion retains merged base pairs that already contain a
three-matching; the resulting search space therefore contains every valid
pair.

The exhaustive counts are
\begin{center}
\begin{tabular}{@{}lrr@{}}
\toprule
bound & rooted pairs & cross-class size excluded\\
\midrule
$\mcz(8,10)\le7$ & $45\cdot114=5130$ & $8$\\
$\mcz(9,10)\le6$ & $27\cdot114=3078$ & $7$\\
$\mc(10,9)\le6$  & $665\cdot114=75810$ & $7$\\
\bottomrule
\end{tabular}
\end{center}
The first two rows prove the two restricted inequalities. The last row
handles exact sizes $9$ and $10$. The only larger case in $\mc(10,9)$ is
$10$ by $10$, already bounded by $2$.
\end{proof}

\subsection{Exclusion of the profiles}

\begin{theorem}
\label{thm:f34-upper}
Every $3$-uniform family with $50$ distinct members contains a
$4$-sunflower. Consequently,
\[
 f(3,4)\le49.
\]
\end{theorem}

\begin{proof}
Assume that $\F$ has $50$ members and contains no $4$-sunflower. Use the
maximum-matching decomposition above. By \cref{prop:f34-profiles}, after
permuting the anchors there are four cases.

For $\{i,j,k\}=\{1,2,3\}$, the subfamily
$\F_j\cup\F_k\cup\C_i$ has matching number exactly $2$. It contains
$T_j,T_k$, while three disjoint members together with $T_i$ would form a
four-matching. Thus $\C_i$ is a cross class to which
\cref{lem:f34-cross-bounds} applies.

In profile I, $N_1=25$, so $\F_1$ is rigid. It has size $8$, and
\cref{lem:f34-interaction}(ii) gives $p(\F_1,T_1)=0$. The class $\C_2$ is
the cross class between $\F_1$ and $\F_3$, of sizes $8$ and $10$. Therefore
\[
 c_2\le\mcz(8,10)\le7,
\]
contrary to $c_2=8$.

In profile II, $N_2=25$, so the size-$9$ block $\F_2$ is rigid and
$p(\F_2,T_2)=0$. The class $\C_1$ is the cross class between $\F_2$ and the
size-$10$ block $\F_3$. Hence
\[
 c_1\le\mcz(9,10)\le6,
\]
contrary to $c_1=7$.

In profiles III and IV, $\C_1$ is the cross class between blocks of sizes
$9$ and $10$. In each case,
\[
 c_1\le\mc(10,9)\le6,
\]
contrary to $c_1=7$. All four profiles are impossible.
\end{proof}

Combining \cref{thm:f34-lower,thm:f34-upper} gives
\[
 39\le f(3,4)\le49.
\]

\subsection{Computational verification}

Four computations enter \cref{thm:f34-upper}: the rooted-block enumeration in
\cref{lem:f34-block}, the eight-vertex cover enumeration in
\cref{lem:f34-interaction}, the profile calculation in
\cref{prop:f34-profiles}, and the cross-family searches in
\cref{lem:f34-cross-bounds}. The supplementary material contains the enumerators, the canonicalizer, the
profile verifier, the search results, and the result checker. The checker reproduces
\[
 55\ \longrightarrow\ 11\ \longrightarrow\ 4
\]
profiles, verifies the cover bounds $3,2,1$ for $r=7,8,9$, and
checks all $5130$, $3078$, and $75810$ rooted-pair results. The enumerators
establish the exclusions; the checker separately verifies that their result
sets cover all required rooted pairs.

\section{Bounds for \texorpdfstring{$f(4,3)$}{f(4,3)}}
\label{sec:f43}

This section proves $54\le f(4,3)\le83$ and determines the intersecting
extremum $\psi(4,3,2)=27$ exactly. For the upper bound, we fix a member $R$
and split the family into the members meeting $R$ and the members disjoint
from $R$. The members disjoint from $R$ pairwise intersect, so their number
is at most $\psi(4,3,2)$. The members meeting $R$ number at most $56$, by
an analysis of the exact trace classes of $R$. Together these bounds give
$f(4,3)\le56+27=83$. The lower bound places two copies of the extremal
$27$-member family on disjoint supports. The upper bound is
computer-assisted at two points. The support bound in
\cref{prop:f43-triple-inputs} rests on fifteen unsatisfiability
certificates, and the inequality $\psi(4,3,2)\le27$ rests on an exhaustive
canonical-augmentation search.

\subsection{Finite triple inputs}

The neighborhood argument uses two finite facts about intersecting triple
families. The first is a size bound. The second concerns the support at
equality.

The following exact value is due to Abbott and Gardner
\cite{AbbottGardner1969}:
\begin{equation}
 f(3,3)=20.
 \label{eq:f33-exact}
\end{equation}

\begin{proposition}[Finite triple bounds]
\label{prop:f43-triple-inputs}
The following statements hold.
\begin{enumerate}[label=\textup{(\roman*)}]
\item Every intersecting, $3$-uniform family without a $3$-sunflower has at
      most $10$ members. In the notation of \cref{eq:psi-definition},
      \[
       \psi(3,3,2)\le10.
      \]
\item If such a family has exactly $10$ members, then its support has at most
      $6$ points.
\end{enumerate}
\end{proposition}

\begin{proof}
Part~(i), due to Abbott and Hanson \cite[p.~9]{AbbottHanson1974}, follows
directly from \cref{eq:split,eq:f33-exact}, which give
\[
 20=f(3,3)\ge2\psi(3,3,2).
\]

We prove part~(ii) by a finite computation. Fix the anchor $\{0,1,2\}$ and
use one Boolean variable for each triple that meets it. A sequential counter
imposes the cardinality of exactly ten selected triples
\cite{Sinz2005}, and the anchor variable is selected. For each pair of
disjoint triples, a two-literal clause forbids selecting both.
For every pair of points, an at-most-two constraint excludes a sunflower with
that pair as core. For every point, a clause excludes each three-matching in
its link. This excludes a sunflower with a one-point core. Pairwise
intersection already excludes the empty-core case. These are all possible
core sizes for three distinct triples. For an exact-support instance, each
point is also required to occur in a selected triple.

An intersecting family of ten triples has support at most
$3+9\cdot2=21$. The fifteen exact-support instances for support sizes
$7,8,\ldots,21$, each with exactly ten selected triples, are unsatisfiable.
The fifteen unsatisfiability proofs were generated by CaDiCaL
\cite{BiereEtAl2024} and accepted by the independent checker
\texttt{drat-trim} \cite{WetzlerHeuleHunt2014}. The CNF encoder generates the
clauses described above directly. The anchor and support bounds show that
the instances cover every family in part~(ii).
\end{proof}
Equation~\eqref{eq:f1-f2} gives
$f(2,3)=6$. An intersecting graph without a $3$-sunflower has at most
three edges: an intersecting simple graph is a star or a triangle, and a
three-edge star is a sunflower. Hence
\begin{equation}
 \psi(2,3,2)=3,
 \qquad
 \psi(1,3,2)=1.
 \label{eq:f43-small-psi}
\end{equation}
The first value is also recorded by Abbott and Hanson
\cite[p.~9]{AbbottHanson1974}.

\subsection{The closed meeting neighborhood}

Let $\F$ be a $4$-uniform family without a $3$-sunflower, and fix
$R\in\F$. Define
\[
 \mathcal N_{\F}[R]=\{F\in\F:F\cap R\ne\varnothing\}.
\]
For $\varnothing\ne T\subsetneq R$, let
\[
 \C_T(R)=\{F\in\F:F\cap R=T\}
\]
be the exact trace class at $T$.

\begin{lemma}[Exact-trace residues]
\label{lem:f43-exact-trace}
The residues $F\setminus T$, $F\in\C_T(R)$, are pairwise distinct and form
an intersecting, $(4-|T|)$-uniform family without a $3$-sunflower. Consequently,
\begin{equation}
 |\C_T(R)|\le\psi(4-|T|,3,2).
 \label{eq:f43-trace-cap}
\end{equation}
\end{lemma}

\begin{proof}
Exactness gives $(F\setminus T)\cap R=\varnothing$. Since $T\subseteq F$
for every $F\in\C_T(R)$, adjoining $T$ recovers $F$; distinct members
therefore give distinct residues. Two disjoint residues, together with $R$,
would lift to a sunflower with core $T$. A sunflower among three residues
would also lift after adjoining $T$.
\end{proof}

For $x\in R$ and $\{x,y\}\in\binom R2$, put
\begin{align*}
 s_x&=|\C_{\{x\}}(R)|,&s&=\sum_{x\in R}s_x,\\
 p_{xy}&=|\C_{\{x,y\}}(R)|,&
 p&=\sum_{\{x,y\}\in\binom R2}p_{xy},&
 r_x&=\sum_{y\in R\setminus\{x\}}p_{xy},\\
 t_m&=|\C_{R\setminus\{m\}}(R)|,&t&=\sum_{m\in R}t_m.
\end{align*}
By \cref{prop:f43-triple-inputs,lem:f43-exact-trace,eq:f43-small-psi},
\begin{equation}
 s_x\le10,
 \qquad p_{xy}\le3,
 \qquad t_m\le1.
 \label{eq:f43-initial-traces}
\end{equation}
Put $h=s+p+t$. The only member with full trace $R$ is $R$ itself, so
\begin{equation}
 |\mathcal N_{\F}[R]|=1+h.
 \label{eq:f43-neighborhood-h}
\end{equation}
The link bound and \cref{eq:f33-exact} give $d_{\F}(x)\le20$ for every
point $x$. Counting incidences with the four points of $R$, apart from the
four incidences contributed by $R$, gives
\begin{equation}
 s+2p+3t\le4(20-1)=76.
 \label{eq:f43-incidence}
\end{equation}

Call $x\in R$ saturated if $s_x=10$, and define
\[
 \Hh_x=\{F\setminus\{x\}:F\in\C_{\{x\}}(R)\},
 \qquad U_x=\supp(\Hh_x).
\]
If $x$ is saturated, \cref{prop:f43-triple-inputs} gives $|U_x|\le6$.
No pair occurs in three members of $\Hh_x$. Counting the pairs in its ten
triples gives
\[
 30\le2\binom{|U_x|}{2}\le30.
\]
Equality holds in both bounds. Thus $|U_x|=6$, every pair of support points has
degree $2$, and every support point has degree $5$. In particular,
$\Hh_x$ is a simple $2$-$(6,3,2)$ design. Five blocks contain any fixed
support point, and two blocks contain any fixed pair. Thus two of the ten
blocks avoid any prescribed pair of support points. In particular, a block
avoids any prescribed set of at most two support points.

\begin{lemma}[Incident pair traces]
\label{lem:f43-incident}
If $x$ is saturated, then $r_x\le6$.
\end{lemma}

\begin{proof}
For $y\in R\setminus\{x\}$, the residues of
$\C_{\{x,y\}}(R)$ are distinct pairwise-intersecting edges. If
$p_{xy}=3$, they cannot form a three-edge star, since their lifts would form
a sunflower. They therefore form a triangle, with support $V_y$.

Consider the complete link at $x$. Every point in this triple family has
degree at most $f(2,3)=6$. A point in $U_x\cap V_y$ would occur five times
in $\Hh_x$ and twice in the $y$-triangle. Hence
\begin{equation}
 U_x\cap V_y=\varnothing
 \qquad\text{when }p_{xy}=3.
 \label{eq:f43-support-disjoint}
\end{equation}
Thus every edge of the $y$-triangle is disjoint from $U_x$.

Suppose that $r_x\ge7$. Write $R\setminus\{x\}=\{y,z,w\}$ and choose the
labels so that $p_{xy}=3$. Then $p_{xz}+p_{xw}\ge4$.

If $p_{xz}=3$ or $p_{xw}=3$, relabel so that $p_{xz}=3$, and let $\Delta$
be the $y$-triangle. Every edge of the
$z$-triangle meets every edge of $\Delta$. Otherwise those two disjoint
edges would allow us to choose a block of $\Hh_x$ avoiding the at most two
points of the $z$-edge that lie in $U_x$. The $y$-edge is disjoint from
$U_x$, so these would be three disjoint triples in the link at $x$. An edge
meeting all three edges of a triangle is an edge of that triangle, so the two
triangles coincide. The $w$-class is nonempty, since $r_x\ge7$. If one of its residue edges
missed an edge of $\Delta$, choose a block of $\Hh_x$ avoiding its at most two
points in $U_x$. This would again give three disjoint link triples. The
residue edge therefore belongs to $\Delta$. It then occurs under all three
labels $y,z,w$, giving a sunflower in the link.

The remaining case is $(p_{xz},p_{xw})=(2,2)$. The same disjointness
argument places both two-edge residue families inside the three edges of
$\Delta$. Two two-element subsets of a three-element set share an edge.
That edge occurs under $y,z,w$ and gives the same contradiction.
\end{proof}

Let $\sigma$ be the number of saturated points of $R$. Then
\[
 s\le10\sigma+9(4-\sigma)=36+\sigma,
 \qquad \sigma\ge s-36.
\]
For a saturated point, \cref{lem:f43-incident} gives $r_x\le6$; for every
other point, \cref{eq:f43-initial-traces} gives $r_x\le9$. Since each pair
trace is counted at both endpoints,
\begin{equation}
 2p=\sum_{x\in R}r_x
 \le6\sigma+9(4-\sigma)
 \le144-3s.
 \label{eq:f43-pair-aggregate}
\end{equation}
Substitute $s=h-p-t$ into
\cref{eq:f43-incidence,eq:f43-pair-aggregate}. This gives
\[
 p+2t\le76-h,
 \qquad
 p+3t\ge3h-144.
\]
Hence $t\ge4h-220$. Since $t\le4$, one has $h\le56$. Equality forces
\begin{equation}
 (s,p,t)=(40,12,4).
 \label{eq:f43-last-profile}
\end{equation}

Excluding this profile improves the bound on
$|\mathcal N_{\F}[R]|=1+h$ from $57$ to $56$, and hence the final upper
bound on $f(4,3)$ from $84$ to $83$.

\begin{lemma}[Exclusion of the equality profile]
\label{lem:f43-last-profile}
The profile in \cref{eq:f43-last-profile} cannot occur.
\end{lemma}

\begin{proof}
Assume \cref{eq:f43-last-profile}. Every singleton class is saturated, and
every triple trace is occupied once. Write
\[
 T_m=(R\setminus\{m\})\cup\{z_m\},
 \qquad z_m\notin R
 \quad(m\in R).
\]
For $x\in R$, define
\[
 \K_x=\{F\setminus\{x\}:x\in F\in\F,\ |F\cap R|\ge2\}.
\]

Every member $K\in\K_x$ contains at least one point of
$R\setminus\{x\}$. The support $U_x$ is disjoint from $R$. Hence
$|K\cap U_x|\le2$. If the intersection is a pair, the two design blocks
through that pair, together with $K$, form a sunflower. If it is a singleton
$\{u\}$, the five design blocks through $u$ induce a $2$-regular simple graph
on the other five support points. This graph is a $5$-cycle. Two disjoint
edges of the cycle give two design blocks meeting exactly in $u$; together
with $K$ they form a sunflower. Thus
\begin{equation}
 K\cap U_x=\varnothing\qquad(K\in\K_x).
 \label{eq:f43-K-avoids}
\end{equation}

The family $\K_x$ is intersecting. Two disjoint members, together with any
block of $\Hh_x$, would be three disjoint triples in the link at $x$.
Counting over all $x\in R$ gives
\[
 \sum_{x\in R}|\K_x|=4+2p+3t=40.
\]
Each $\K_x$ is an intersecting triple family without a $3$-sunflower, so
\cref{prop:f43-triple-inputs} bounds it by $10$. Hence every $\K_x$ has
size $10$. By \cref{prop:f43-triple-inputs}(ii), its support has at most six
points. No pair occurs in three members of $\K_x$, since those members would
form a sunflower with that pair as core. The ten triples contain $30$ pair
occurrences. Therefore
\[
 30\le2\binom{|\supp(\K_x)|}{2}\le30.
\]
Equality holds in both bounds. Thus the support has exactly six points, every
pair of support points has degree $2$, and $\K_x$ is a simple
$2$-$(6,3,2)$ design.

Fix $x$ and put $Y=R\setminus\{x\}$. The design $\K_x$ contains
\[
 Y,
 \qquad
 B_m=(Y\setminus\{m\})\cup\{z_m\}
 \quad(m\in Y).
\]
Its support has six points, so the three $z_m$, $m\in Y$, lie in the
three-point complement of $Y$. They are distinct. Indeed, suppose that
$z_i=z_j=q$ and let $y$ be the third point of $Y$. The pairs
$\{y,i\}$, $\{y,j\}$, and $\{y,q\}$ already occur twice among the displayed
blocks. The two remaining design blocks through $y$ must avoid $i,j,q$.
Both would therefore consist of $y$ and the other two outside support points,
contradicting distinctness.

Write $Y=\{y,a,b\}$. The two design blocks through $y$ not displayed above
come from $\C_{\{x,y\}}(R)$. The pair $\{y,z_y\}$ occurs in none of the
displayed blocks, so both new blocks contain $z_y$. The pairs
$\{y,z_a\}$ and $\{y,z_b\}$ each occur once and each needs one further
occurrence. Their residue edges are therefore
\begin{equation}
 \bigl\{\{z_y,z_a\},\{z_y,z_b\}\bigr\}.
 \label{eq:f43-x-view}
\end{equation}
For any two indices in $R$, choose $x$ outside them. The preceding argument
shows that their $z$-points are distinct. Thus all four points $z_m$ are
distinct.

Finally fix distinct $x,y$ and write $R\setminus\{x,y\}=\{a,b\}$. The
$x$-view of the exact pair class is \cref{eq:f43-x-view}. Its $y$-view is
\[
 \bigl\{\{z_x,z_a\},\{z_x,z_b\}\bigr\}.
\]
Both views describe
\[
 \{F\setminus\{x,y\}:F\in\C_{\{x,y\}}(R)\},
\]
so they must be equal. They cannot be equal because the four $z$-points are
distinct.
\end{proof}

\begin{proposition}[Meeting-neighborhood bound]
\label{prop:f43-neighborhood}
For every $R\in\F$,
\begin{equation}
 |\mathcal N_{\F}[R]|\le56.
 \label{eq:f43-N56}
\end{equation}
\end{proposition}

\begin{proof}
The preceding inequalities give $h\le56$, and
\cref{lem:f43-last-profile} excludes equality. Thus $h\le55$.
Apply \cref{eq:f43-neighborhood-h}.
\end{proof}

\subsection{The intersecting extremum}

To determine $\psi(4,3,2)$, it remains to prove its upper bound. This follows
from an exhaustive enumeration of the objects defined below.

The decision problem asks whether there is a family with exactly $28$
distinct members, called rows below, satisfying the following conditions:
\begin{enumerate}
\item every row has size four;
\item every two rows intersect;
\item no three rows have equal pairwise intersections.
\end{enumerate}
The ground set is not fixed. If such a family existed, its support would
have size at most
\begin{equation}
 4+27\cdot3=85,
 \label{eq:f43-support85}
\end{equation}
because every row after the first meets the first row.

A partial family is represented by its two-colored bipartite incidence
graph. Row vertices and point vertices are the two color classes. The first
row is fixed as $\{0,1,2,3\}$. If a parent uses the point labels
$0,\ldots,C-1$, a candidate row is generated by choosing an $s$-subset of
the old labels, for $0\le s\le4$, and adjoining the fresh consecutive labels
\[
 C,C+1,\ldots,C+(4-s)-1.
\]
This list contains every extension up to relabelling points not used by the
parent.

A candidate is rejected if it repeats a row, misses a prior row, or completes
a $3$-sunflower. For prior rows $A,B$ and candidate $S$, the last test is
\[
 A\cap B\subseteq S,
 \qquad S\cap A\subseteq B,
 \qquad S\cap B\subseteq A.
\]
These three inclusions are equivalent to
$A\cap B=A\cap S=B\cap S$. All rejection conditions are hereditary under
row deletion.

The program uses nauty to compute canonical labels, vertex orbits, and
automorphisms of the colored incidence graph
\cite{McKay1998,McKayPiperno2014}. Parent automorphisms reduce the candidate
rows to orbit representatives. After a candidate is appended, the program
selects the last row in the canonical row order. The child is accepted when
the appended row belongs to the same full child-automorphism orbit as this
canonical row. Deleting a row in that orbit, and then deleting points used
only by that row, defines the canonical parent.

\begin{lemma}[Canonical construction coverage]
\label{lem:f43-canonical-coverage}
Every isomorphism class of valid partial families occurs on an accepted
construction path. Every valid $28$-member family, if one exists, lies below
an emitted valid depth-$6$ root.
\end{lemma}

\begin{proof}
Proceed by induction on the number of rows. The one-row family has the fixed
representative $\{0,1,2,3\}$. Let $\mathcal G$ be a valid child. Delete a
row in the orbit selected by its canonical labeling, together with points
used only by that row. The resulting canonical parent $\mathcal P$ remains
valid.

By induction, a normalized copy of $\mathcal P$ occurs. The deleted row is
an old-label subset together with fresh points, so an isomorphic candidate is
generated. Parent-orbit reduction retains a representative of its orbit,
and the appended row lies in the canonical deletion orbit of the resulting
child. The child is therefore accepted.

The implementation stores at most $400$ generator permutations.
The stored automorphisms generate a subgroup of the full parent automorphism
group, and subgroup orbits refine full orbits. The limit may duplicate work,
but it preserves at least one representative of every extension. The
canonical-deletion test uses nauty's full orbit output rather than the stored
generator list.

Every accepted object deeper than six rows has an accepted depth-$6$
ancestor obtained by repeated canonical deletion. This proves the second
statement.
\end{proof}

The implementation permits $128$ point columns, $56$ rows, and $184$
vertices in the incidence graph. A target needs at most $28+85=113$
vertices, so these capacities do not restrict the decision problem.

The depth-$6$ generation visits $67{,}193$ search-tree nodes and produces
$11{,}720$ roots. Exhausting the subtrees below these roots visits
$105{,}917{,}089{,}577$ search-tree nodes and finds no $28$-member family. A
separate C verifier reconstructs the depth-$6$ frontier and checks that the
results contain one completed record for every root, with the stated total
node count.

\begin{proposition}[Intersecting upper bound]
\label{prop:f43-intersecting-upper}
Every intersecting, $4$-uniform family without a $3$-sunflower has at most
$27$ members.
\end{proposition}

\begin{proof}[Computer-assisted proof]
Suppose that a $28$-member family existed. Equation~\eqref{eq:f43-support85}
places it within the program capacities. By
\cref{lem:f43-canonical-coverage}, an isomorphic copy lies below one of the
$11{,}720$ depth-$6$ roots. The exhaustive search of every such subtree
finds no target. This is a contradiction. A larger family would contain a
$28$-member subfamily.
\end{proof}

For the reverse inequality, partition nine points into three disjoint
three-point sets $X_0,X_1,X_2$. For every unordered pair
$\{i,j\}\subseteq\{0,1,2\}$, take all sets
\begin{equation}
 A_i\cup A_j,
 \qquad
 A_i\in\binom{X_i}{2},
 \quad A_j\in\binom{X_j}{2}.
 \label{eq:f43-AHS27}
\end{equation}
There are $3\cdot3\cdot3=27$ such sets. This is the construction of Abbott
and Hanson \cite[p.~9]{AbbottHanson1974}.

\begin{proposition}[Intersecting extremum]
\label{prop:f43-intersecting-exact}
One has
\begin{equation}
 \psi(4,3,2)=27.
 \label{eq:f43-psi27}
\end{equation}
\end{proposition}

\begin{proof}
Two members of the family in \cref{eq:f43-AHS27} use two of the three point
sets, so their types share one point set. Any two pairs in a three-point set
meet. The family is therefore intersecting.

Consider three members. If their three point-set types are equal, equal
pairwise intersections in each used point set would force all three chosen
pairs to be equal. The three members would then be equal. If exactly two
types are equal, one of the two point sets used by those members is not used
by the third. The two members intersect in that point set, whereas neither
of the other two intersections contains a point from it. If the types are
$X_0X_1$, $X_0X_2$, and $X_1X_2$, their
three pairwise intersections are nonempty subsets of $X_0,X_1,X_2$,
respectively. They are not equal. Thus the family contains no
$3$-sunflower.

This proves the lower bound $27$. The reverse inequality is
\cref{prop:f43-intersecting-upper}.
\end{proof}

\subsection{The global bounds}

Take two copies of the family in \cref{eq:f43-AHS27} on disjoint supports.
Their union has $54$ members. Three members chosen from one copy do not form
a sunflower. If three members use both copies, two lie in the same copy and
meet, whereas their intersections with a member of the other copy are empty.
Thus their three pairwise intersections are not equal.

\begin{theorem}
\label{thm:f43-bounds}
One has
\begin{equation}
 54\le f(4,3)\le83.
 \label{eq:f43-bounds}
\end{equation}
\end{theorem}

\begin{proof}
The preceding two-copy construction gives the lower bound. For the upper
bound, let $\F$ be a $4$-uniform family without a $3$-sunflower and choose
$R\in\F$. Put
\[
 D_{\F}(R)=\{F\in\F:F\cap R=\varnothing\}.
\]
If two members of $D_{\F}(R)$ were disjoint, those two members and $R$ would
form a sunflower with empty core. Hence $D_{\F}(R)$ is intersecting. By
\cref{prop:f43-neighborhood,prop:f43-intersecting-exact},
\[
 |\F|=|\mathcal N_{\F}[R]|+|D_{\F}(R)|\le56+27=83.
\]
\end{proof}

\subsection{Computational verification}

Two finite computations enter the upper bound in \cref{thm:f43-bounds}.
First, part~(ii) of \cref{prop:f43-triple-inputs} uses fifteen CNF instances
whose unsatisfiability proofs are checked by \texttt{drat-trim}; the anchor
and support bounds in its proof establish coverage. Second,
\cref{prop:f43-intersecting-upper} uses canonical augmentation, with coverage
proved in \cref{lem:f43-canonical-coverage}, to exhaust all $11{,}720$
depth-$6$ roots. The supplementary material contains the search source, the manifests, the
results, the frontier verifier, and the control programs.
The search establishes the exclusion below each root. The verifier
reconstructs the frontier and checks one completed result for each root and
the total node count.

\section{Lower bounds for six and seven petals}
\label{sec:triple-lower-six-seven}

This section proves $f(3,6)\ge153$ and $f(3,7)\ge259$. Combined with
\cref{cor:triple-upper-values}, these bounds give $153\le f(3,6)\le255$ and
$259\le f(3,7)\le474$. Both constructions place two families on disjoint
supports: a first component with matching number at most $k-2$, and an
intersecting second component. For six petals the first component is a
cyclic family of $126$ triples on $14$ points; for seven petals it is a
family of $217$ triples on $17$ points. The second components have $27$ and
$42$ members, respectively.

Equation~\eqref{eq:FW-intersecting-triples} gives
\[
 \psi(3,6,2)=27,
 \qquad
 \psi(3,7,2)=42.
\]
Thus the $27$- and $42$-member intersecting families used as the second
components below are optimal. Within these two-component disjoint-support
constructions, any further improvement must come from the first component.

\subsection{Six petals}

Work in $\Z_{14}$. Let $\Hh$ consist of all translates modulo $14$ of the
nine triples
\begin{equation}
\begin{split}
 &\{0,1,5\},\ \{0,1,6\},\ \{0,1,7\},\ \{0,1,9\},\ \{0,1,10\},\\
 &\{0,10,12\},\ \{0,2,6\},\ \{0,2,9\},\ \{0,2,10\}.
\end{split}
\label{eq:f36-seeds}
\end{equation}

\begin{lemma}
\label{lem:f36-cyclic-family}
The family $\Hh$ has $126$ members. It is $27$-regular, has matching
number at most $4$, and contains no $6$-sunflower.
\end{lemma}

\begin{proof}
The stabilizer of a three-element subset of $\Z_{14}$ has order dividing
both $3$ and $14$. It is therefore trivial, so every seed has an orbit of
length $14$. The cyclic gap sequences of the nine seeds, up to cyclic
rotation, are
\[
\begin{gathered}
 (1,4,9),\ (1,5,8),\ (1,6,7),\ (1,8,5),\ (1,9,4),\\
 (10,2,2),\ (2,4,8),\ (2,7,5),\ (2,8,4).
\end{gathered}
\]
They are distinct. Hence the nine orbits are disjoint and
$|\Hh|=9\cdot14=126$.

Five disjoint triples would require $15$ points. Thus
$\nu(\Hh)\le4$. To exclude two-point cores, count the pairs in the nine
seeds according to cyclic distance. Translation gives the following pair
degrees. A pair at cyclic distance $7$ is fixed by the translation by $7$,
so its orbit has length $7$, and each seed pair at that distance contributes
twice to the degree.
\begin{center}
\begin{tabular}{@{}lrrrrrrr@{}}
\toprule
cyclic distance & $1$ & $2$ & $3$ & $4$ & $5$ & $6$ & $7$\\
\midrule
pair degree in $\Hh$ & $5$ & $5$ & $0$ & $5$ & $5$ & $5$ & $4$\\
\bottomrule
\end{tabular}
\end{center}
No pair is therefore the core of a $6$-sunflower.

For a one-point core, consider $\Lk_{\Hh}(0)$. The vertices $3$ and $11$
are isolated because pairs at cyclic distance $3$ have degree zero. Every
edge of the link lies on the other $11$ vertices, so its matching number is
at most $5$. Translation gives the same conclusion for every point. This
excludes all possible core sizes.

Translation also makes all point degrees equal. Incidence counting gives
\[
 14d_{\Hh}(0)=3|\Hh|=378,
\]
so every point has degree $27$.
\end{proof}

The construction attains the maximum possible size on a $14$-point support.

\begin{proposition}
\label{prop:f36-fourteen-points}
Every $3$-uniform family on at most $14$ points that contains no
$6$-sunflower has at most $126$ members.
\end{proposition}

\begin{proof}
Let $\A$ be such a family. For every point $x$, the link graph
$\Lk_{\A}(x)$ has maximum degree at most $5$ and matching number at most
$5$. A violation gives a $6$-sunflower with core of size two or one,
respectively. By the Chv\'atal--Hanson formula in \eqref{eq:CH},
\[
 |E(\Lk_{\A}(x))|\le\CH(5,5)=27.
\]
Summing the point degrees gives
\[
 3|\A|\le14\cdot27=378.
\]
Thus $|\A|\le126$.
\end{proof}

For a second component, take the graph $G$ on a disjoint copy of $\Z_{11}$
with edge set
\begin{equation}
 \bigl\{\{i,i+1\},\{i,i+2\}:i\in\Z_{11}\bigr\}
 \mathbin{\cup}
 \bigl\{\{i,i+5\}:0\le i\le4\bigr\},
 \label{eq:f36-graph}
\end{equation}
where addition in the first set is modulo $11$. The graph has $27$ edges,
maximum degree $5$, and matching number at most $5$. It therefore contains
no $6$-sunflower. Add a new point $c$ and put
\[
 \K=\bigl\{\{c\}\cup e:e\in E(G)\bigr\}.
\]
The family $\K$ contains no $6$-sunflower, has $27$ members, and has matching
number $1$.

\begin{theorem}
\label{thm:f36-bounds}
One has
\[
 153\le f(3,6)\le255.
\]
\end{theorem}

\begin{proof}
Place $\Hh$ and $\K$ on disjoint supports. A sunflower with nonempty core
cannot use members from both families. An empty-core sunflower is a
matching, while
\[
 \nu(\Hh\dd\K)=\nu(\Hh)+\nu(\K)\le4+1=5.
\]
Hence $\Hh\dd\K$ contains no $6$-sunflower and has $126+27=153$ members.
The upper bound is \cref{cor:triple-upper-values}.
\end{proof}

\subsection{Seven petals}

The next construction uses a finite triple system. The family and a
standalone verifier are part of the supplementary material.

The verifier computes graph matching numbers by the following direct
recurrence. For a graph $G$ and $S\subseteq V(G)$, put
$\mu_G(S)=\nu(G[S])$ and $\mu_G(\varnothing)=0$. For nonempty $S$, choose
$v\in S$. Then
\begin{equation}
 \mu_G(S)=\max\left(
 \{\mu_G(S\setminus\{v\})\}
 \mathbin{\cup}
 \bigl\{1+\mu_G(S\setminus\{u,v\}):u\in N_G(v)\cap S\bigr\}
 \right).
 \label{eq:f37-matching-recurrence}
\end{equation}
The two terms distinguish whether a maximum matching leaves $v$ unmatched or
matches it to a neighbor. Memoization over vertex subsets makes the
calculation exact.

\begin{lemma}
\label{lem:f37-base-family}
There is a family $\A$ of $217$ triples on the point set
$\{0,1,\ldots,16\}$ such that every pair has degree at most $6$ and every
vertex link has matching number at most $6$. The family contains no
$7$-sunflower and satisfies $\nu(\A)\le5$.
\end{lemma}

\begin{proof}
Take the family listed in the supplementary material. The verifier checks
that it contains $217$ distinct sorted triples and has support
$\{0,1,\ldots,16\}$. It enumerates all pair degrees and applies
\eqref{eq:f37-matching-recurrence} to every vertex link. The resulting link
data are
\begin{center}
\begingroup
\scriptsize
\setlength{\tabcolsep}{3.5pt}
\begin{tabular}{c|rrrrrrrrrrrrrrrrr}
\toprule
vertex &0&1&2&3&4&5&6&7&8&9&10&11&12&13&14&15&16\\
\midrule
$|\Lk_{\A}(v)|$&38&39&39&39&38&39&39&38&35&38&37&39&38&39&38&39&39\\
$\nu(\Lk_{\A}(v))$&6&6&6&6&6&6&6&6&6&6&6&6&6&6&6&6&6\\
\bottomrule
\end{tabular}
\endgroup
\end{center}

A sunflower of triples has core size zero, one, or two. A two-point core
would require pair degree at least $7$. A one-point core would require a
matching of size $7$ in a vertex link. Seven disjoint triples cannot fit on
$17$ points. These facts exclude the three cases. The point count also
gives $\nu(\A)\le\lfloor17/3\rfloor=5$.
\end{proof}

Take two disjoint seven-point sets $P,Q$ and a new point $z$, all outside the
support of $\A$. Let $K_P$ and $K_Q$ be the complete graphs on $P$ and $Q$,
respectively, and put
\[
 H=K_P\dd K_Q,
 \qquad
 \B=\bigl\{\{z,x,y\}:xy\in E(H)\bigr\}.
\]
The graph $H$ has maximum degree $6$, matching number $6$, and $42$ edges.
It contains no $7$-sunflower. Deleting $z$ shows that $\B$ also contains
no $7$-sunflower. Moreover, $|\B|=42$ and $\nu(\B)=1$.

\begin{theorem}
\label{thm:f37-bounds}
One has
\[
 259\le f(3,7)\le474.
\]
\end{theorem}

\begin{proof}
The supports of $\A$ and $\B$ are disjoint. A sunflower with nonempty core
cannot use triples from both families. An empty-core sunflower is a
matching, while
\[
 \nu(\A\dd\B)=\nu(\A)+\nu(\B)\le5+1=6.
\]
Thus $\A\dd\B$ contains no $7$-sunflower and has $217+42=259$ members.
The upper bound is \cref{cor:triple-upper-values}.
\end{proof}

\subsection{Verification of the finite witness}

The certificate for the lower bound $f(3,7)\ge259$ consists of the explicit
$217$-triple family and the $42$-member family $\B$ constructed above. The
standalone verifier checks the size and distinctness of the $217$-triple
family, its support, every pair degree, and every vertex-link matching
number.

\section{Formal verification in Lean}
\label{sec:lean}

The main combinatorial arguments of
\cref{sec:preliminaries,sec:constructions,sec:triple-upper}, the appendix
capacity recurrence, and every explicit finite witness used for the numerical
lower bounds have been formalised in the Lean~4 proof assistant
\cite{deMouraUllrich2021} over the Mathlib library \cite{mathlib2020}. Since
the finiteness of $f(w,k)$ rests on the Erd\H{o}s--Rado theorem, the
formalisation represents $f$ and $\psi$ by pairs of upper-bound and witness
predicates. The finite witness families are encoded explicitly in Lean,
and their properties are verified by computation within Lean's kernel.

Some external results and analytic arguments used in the paper are outside
the Lean formalisation. The Chv\'atal--Hanson upper bound is supplied as a
hypothesis. For the three exponential-rate bounds, Lean verifies
constructions of size at least $c\,r_k^w$, for some $c>0$ and every $w\ge3$;
the limiting and asymptotic arguments are not included. The exhaustive
computations behind $f(3,4)\le49$ and $f(4,3)\le83$ are checked separately,
as described in \cref{sec:f34,sec:f43}. The correspondence between the paper statements and
formal theorems, together with the pinned toolchain and build instructions,
is part of the supplementary material.

\appendix

\section{Alternative capacity recurrence}
\label{sec:capacity-appendix}

This appendix proves a second lower-bound recurrence for $\psi(w,k,m)$,
valid for every matching cap $2\le m\le k-1$. The construction assigns a
smaller sunflower-free family to each edge of a bounded-degree auxiliary
graph on $m$ points; bounding the vertex degrees limits how often any
auxiliary point is used. At $m=k-1$ the resulting bound coincides with
\cref{cor:endpoint-recurrence}, so no result in the body of the paper
depends on this appendix. At smaller matching caps neither recurrence
dominates the other. For $k=5$ and $m=3$, \cref{thm:block-size} gives
$\psi(w,5,3)\ge4\psi(w-1,5,3)+f(w-3,5)$, while \cref{thm:capacity} gives
$\psi(w,5,3)\ge3\psi(w-1,5,3)+3f(w-2,5)$; which bound is larger depends on
the relative growth of the two quantities.

We use the following standard circulant construction.

\begin{lemma}[Bounded-degree graph]
\label{lem:bounded-degree-graph}
For integers $n\ge1$ and $0\le d\le n-1$, there is a graph on $n$ vertices
with maximum degree at most $d$ and exactly $\lfloor nd/2\rfloor$ edges.
\end{lemma}

\begin{proof}
Use vertex set $\Z_n$. If $d=0$, take the empty graph. If $d$ is positive
and even, join vertices at cyclic distances $1,\ldots,d/2$. The resulting
graph is $d$-regular.

Suppose that $d$ is odd and $n$ is even. Join vertices at cyclic distances
$1,\ldots,(d-1)/2$ and add the antipodal perfect matching. This graph is
$d$-regular.

Suppose that both $n$ and $d$ are odd. Then $d\le n-2$. Start with the
$(d-1)$-regular graph using cyclic distances $1,\ldots,(d-1)/2$. The edges
of cyclic distance $(n-1)/2$ form one cycle of length $n$, since
\[
 \gcd\left(\frac{n-1}{2},n\right)=1.
\]
Add a matching of size $(n-1)/2$ from this cycle. The resulting graph has
$n-1$ vertices of degree $d$ and one vertex of degree $d-1$. It has
$(nd-1)/2=\lfloor nd/2\rfloor$ edges.
\end{proof}

\begin{theorem}[Capacity recurrence]
\label{thm:capacity}
Let $k\ge3$ and $2\le m\le k-1$. Put
\[
 d=\min\{k-m,m-1\},
 \qquad
 E(m,k)=\left\lfloor\frac{md}{2}\right\rfloor.
\]
For every $w\ge3$,
\begin{equation}
 \psi(w,k,m)
 \ge m\psi(w-1,k,m)+E(m,k)f(w-2,k).
 \label{eq:capacity}
\end{equation}
\end{theorem}

\begin{proof}
Take pairwise disjoint sets $X,Y,Z$, where $|Y|=m$. On $X$, take a
$(w-1)$-uniform family $\A$ with no $k$-sunflower such that
\[
 |\A|=\psi(w-1,k,m),
 \qquad
 \nu(\A)\le m-1.
\]
By \cref{lem:bounded-degree-graph}, there is a graph $P$ on $Y$ with
\[
 \Delta(P)\le d,
 \qquad
 |E(P)|=E(m,k).
\]
For each $p\in E(P)$, take a $(w-2)$-uniform family $\E_p$ on $Z$ with no
$k$-sunflower and with $|\E_p|=f(w-2,k)$. Define
\begin{equation}
\begin{split}
 \F={}&\{T\cup\{y\}:T\in\A,\ y\in Y\}\\
 &{}\cup\{p\cup S:p\in E(P),\ S\in\E_p\}.
\end{split}
\label{eq:capacity-family}
\end{equation}
The two lines meet $Y$ in one and two points. They are disjoint, and all
members within each line are distinct. This gives the count in
\cref{eq:capacity}.

Let a matching in $\F$ contain $q$ members from the first line and $b$ from
the second. If $b=0$, the $X$-parts form a matching in $\A$, so
$q\le m-1$. If $b\ge1$, disjointness of their intersections with $Y$ gives
$q+2b\le m$. Hence
\[
 q+b\le m-b\le m-1.
\]
Thus $\nu(\F)\le m-1$.

Suppose that $k$ members of $\F$ form a sunflower. Every member meets $Y$,
and $|Y|=m<k$. By \cref{lem:occurrence}, some $y\in Y$ belongs to all $k$
members.

If all members come from the first line of \cref{eq:capacity-family}, deleting
$y$ gives a $k$-sunflower in $\A$. Suppose that they all come from the second
line. Every corresponding edge of $P$ contains $y$. If a second endpoint
occurs in two members, the occurrence rule forces it to occur in all $k$
members. All members then use the same edge $p$, and deleting $p$ gives a
$k$-sunflower in $\E_p$. If no second endpoint repeats, the number of members
is at most
\[
 \deg_P(y)\le d\le k-m<k.
\]

Finally, suppose that both lines occur. The intersection of a member from the
first line with one from the second is exactly $\{y\}$. This is the sunflower
core. The $X$-parts of the members from the first line form a matching in
$\A$, so there are at most $m-1$ of them. A second endpoint of an edge of $P$
cannot occur twice among the second-line members: it occurs in no member from
the first line, contrary to the occurrence rule. There are therefore at
most $\deg_P(y)\le k-m$ members from the second line. The total is at most
\[
 (m-1)+(k-m)=k-1,
\]
which is impossible.
\end{proof}

At $m=k-1$, \cref{thm:block-size,thm:capacity} both give
\cref{eq:endpoint-recurrence}. For other values of $m$, their coefficients
and lags differ.

\section*{Acknowledgments}

This paper was produced in an AI-assisted workflow. AI assistants helped
explore candidate constructions, draft proofs, and write software. The
authors verified all statements and results and remain fully responsible for
them.

\end{document}